\documentclass[journal,onecolumn,11pt]{IEEEtran}

\ifCLASSINFOpdf
\else
\fi

\usepackage{amsmath}
\usepackage{amssymb}
\usepackage{graphicx}
\usepackage{color}
\usepackage{dsfont,latexsym}
\usepackage{tikz}
\usepackage{algorithm,algpseudocode}
\usepackage{enumitem}

\newtheorem{lemma}{Lemma}
\newtheorem{remark}{Remark}
\newtheorem{theorem}{Theorem}

\newtheorem{proposition}{Proposition}

\newtheorem{definition}{Definition}

\allowdisplaybreaks

\begin{document}
\title{On hyperbolic PDEs, filtered feedback control laws, and fractal-like stability crossing curves}

 \author{Wim Michiels, Federico Bribiesca-Argomedo, and Jean Auriol
\thanks{Wim Michiels is with the Department of Computer Science, KU Leuven, B-3001 Heverlee, Belgium, wim.michiels@cs.kuleuven.be.}
\thanks{Federico Bribiesca-Argomedo is with  INSA Lyon, Universite Claude Bernard Lyon 1, Ecole Centrale de Lyon, CNRS, Ampère, UMR5005, 69621 Villeurbanne, France,  federico.bribiesca-argomedo@insa-lyon.fr}
 \thanks{J. Auriol is with Université Paris-Saclay, CNRS, CentraleSupélec, Laboratoire des Signaux et Systèmes, 91190 Gif-sur-Yvette, France, jean.auriol@centralesupelec.fr.} 
 }

\markboth{}{S}%

\maketitle

\begin{abstract}
The paper addresses the boundary control of a class of hyperbolic PDEs, based on an equivalent representation in terms of an integral-difference equation. The situation is considered where direct compensation of reflection terms induces a fragile closed-loop system, in the sense of lack of strong stability. This is theoretically resolved by adding a low-pass filter to the control law, but the choice of its cut-off frequency is crucial in balancing robustness at high frequencies and performance at low frequencies. First, the maximum stability interval in parameter $T$ is determined, with $T$ the inverse of the filter's cutoff frequency. Next, model mismatch on the PDE parameters is considered and a sufficient stability condition is derived in terms of allowable mismatch and cut-off frequency, satisfied in a region in the combined parameter space with a conic shape around $T=0$. Finally, this qualitative behavior is confirmed by exact stability charts for a special case where all model mismatch is contained into one parameter.  It is highlighted that the set of stability crossing curves exhibits a fractal-like structure, which is explained using a limit system with discrete delays.
\end{abstract}

\begin{IEEEkeywords}
Hyperbolic systems, integral difference equations, filtering, robustness margins
\end{IEEEkeywords}

% \fi
\IEEEpeerreviewmaketitle

\section{Introduction}

\subsection{System under consideration}
{W}e consider in this paper the boundary control of general linear balance laws, which naturally arise in the modeling of physical phenomena involving transport or propagation~\cite{bastin2016stability}. Such systems appear in a wide range of applications, including drilling devices~\cite{auriol2022comparing,auriol2020closed}, water management networks~\cite{diagne2017control}, aeronomy~\cite{schunk1975transport}, and traffic flow systems~\cite{espitia2022traffic}. More precisely, linear hyperbolic balance laws can be modeled as
\begin{equation} \label{eq:hyperbolic_couple}
\left\{
\begin{aligned}
&z_t(t, y)+\Lambda z_y(t, y)=\Sigma z(t,y), \quad t>0,~ y \in[0,1], \\
&z_+(t, 0)=Q z_-(t, 0),\quad z_-(t, 1)=R z_+(t, 1)+u(t), %\quad t>0,
\end{aligned}
\right.
\end{equation}
where $z(t,y)=(z^\top_+(t,y),z^\top_-(t,y))^\top \in\mathbb{R}^{k+\ell}$ is the state, with $k,\ell \in \mathbb{N}$. 
The components $z_+ \in\mathbb{R}^{k}$ and $z_- \in\mathbb{R}^{\ell}$ represent, respectively, the rightward and leftward propagating states. 
The diagonal matrix $\Lambda$ is given by $\Lambda=\mathrm{diag}(\lambda_1,\dots,\lambda_{k},-\mu_1,\dots,-\mu_\ell),$
with ordered velocities
\begin{equation} \label{eq:original_order}
-\mu_\ell<\cdots<-\mu_1<0<\lambda_1<\cdots<\lambda_{k}.
\end{equation}
The constant matrix $\Sigma$ accounts for in-domain couplings and is decomposed as
\begin{equation}
    \Sigma=\begin{pmatrix}
        \Sigma^{++} &  \Sigma^{+-}\\[2mm]  
        \Sigma^{-+} &  \Sigma^{--}
    \end{pmatrix},
\end{equation}
with $\Sigma^{++} \in \mathbb{R}^{k\times k}$, $\Sigma^{+-} \in \mathbb{R}^{k\times \ell}$, $\Sigma^{-+} \in \mathbb{R}^{\ell\times k}$, and $\Sigma^{--} \in \mathbb{R}^{\ell\times \ell}$. Note that the results of this paper can be extended to continuously differentiable functions $\Sigma^{\cdot \cdot}$. Without any loss of generality~\cite{hu2015control}, we assume that $\Sigma_{ii}=0$ for all $1\leq i\leq k+\ell$.
The constant matrices $Q\in\mathbb{R}^{k\times \ell}$ and $R\in\mathbb{R}^{\ell\times k}$ correspond to boundary couplings, while the control input is denoted {$u(t)\in \mathbb{R}^{\ell}$.} 
The initial conditions $(u_0,v_0)$ are assumed to belong to $(H^1([0,1],\mathbb{R}))^{k}$ and $(H^1([0,1],\mathbb{R}))^{\ell}$, respectively. They satisfy the classical zero-order compatibility conditions, so that the open-loop system is well-posed~\cite{bastin2016stability}. 
We recall the following definition for $L^2$-stability.
 %For continuous control inputs, the system~\eqref{eq:hyperbolic_couple} is well-posed~\cite{bastin2016stability} 
%{\color{blue}H1 or L2?}
\begin{definition} \label{Def_stability}
    The zero solution of~\eqref{eq:hyperbolic_couple} is exponentially stable if there exists $\kappa_0>0$ and $\eta_0>0$ such that for any initial condition  $(u_0,v_0) \in (L^2([0,1],\mathbb{R}))^{k+\ell}$, %that satisfies the zero-order compatibility conditions $u_0(0)=Qv_0(0)$ and $v_0(1)=Ru_0(1)$
     we have
    \begin{align}
        ||(u,v)||_{L^2} \leq \kappa_0 \mathrm{e}^{-\eta_0t}||(u_0,v_0)||_{L^2}. 
    \end{align}
\end{definition}

Stabilizing controllers for the system~\eqref{eq:hyperbolic_couple} have been proposed in the literature using Lyapunov-based approaches~\cite{bastin2016stability} or the backstepping methodology~\cite{coron2017finite}.
Comprehensive overviews of current research in this field can be found in~\cite{hayat2021boundary}, and \cite{vazquez2024backstepping}.

%The backstepping approach has enabled significant breakthroughs for the stabilization of such systems. 
\subsection{Time-delay representation}
The connection between hyperbolic systems and time-delay systems was first established through D'Alembert's formula~\cite{DAlembert}. This relation extends via the method of characteristics, which allows certain classes of linear conservation laws to be represented as difference equations. In~\cite{russell1991neutral}, the existence of a mapping between the solutions of first-order hyperbolic partial differential equations and difference equations is demonstrated using spectral methods. This mapping is shown to be unique, and an explicit construction is provided in~\cite{karafyllis2014relation} for a single hyperbolic equation with a reaction term. This result was later generalized in~\cite{Auriol2019a} for systems of balance laws. The analysis and control design developed in this paper rely on such a comparison model for~(\ref{eq:hyperbolic_couple}), expressed in the form of the integral-difference equation (IDE) 
\begin{align}
x(t)=\sum_{i=1}^k \sum_{j=1}^{\ell} H_{ij} x\left(t-(r_i+s_j)\right)+\int_{-\tau}^0 w_1(\theta) x(t+\theta)d\theta  +u(t),\label{sys-final}
\end{align}
where {$x(t)\in \mathbb{R}^\ell$}, matrices $H_{ij}$ and the piecewise continuously differentiable kernel function $w_1$ are determined by the PDE parameters and the delays satisfy {
\[
r_i=1/\lambda_i,\ 1\leq i\leq k,\ \  s_j=1/\mu_j,\ 1\leq j\leq \ell,\ \  \tau=\frac{1}{\lambda_1}+\frac{1}{\mu_1}.
\]
For the derivation of this comparison model, rooted in integral transformations, we refer to~\cite{Auriol2019a} or Appendix~\ref{apA}. {Note that the fact that the dimension of $x$ is equal to $\ell$ in the IDE~\eqref{sys-final} is a direct consequence of the PDE formulation from which the system originates. Nevertheless, all the results presented in this paper remain valid for an IDE whose state $x(t)$ has an arbitrary dimension $m \in \mathbb{N}$}. 
{For IDE~\eqref{sys-final} let $x_s$ denote the value of the state $x$ over the time interval $t\in[s-\tau, s]$, with associated norm $||x_s||_{L^2_\tau}=\left(\int_{-\tau}^{0}\vert\vert x(s+r)\vert\vert_2^2 dr\right)^{1/2}$.} 
The exponential stability of (\ref{sys-final}) is defined as follows.
\begin{definition}\label{def_exp_stab_IDE}
    The zero solution of  IDE system~\eqref{sys-final} is $L^2_\tau$ exponentially stable if there exists $\kappa_0>0$ and $\eta_0>0$ such that for any initial condition $x_\tau \in L^2_\tau$, we have for all $t>\tau$
\begin{align}
||x_t||_{L^2_\tau} \leq \kappa_0\mathrm{e}^{-\eta_0 t}||x_\tau||_{L^2_{\tau}}
\end{align}
\end{definition}
Systems  \eqref{eq:hyperbolic_couple} and (\ref{sys-final}) have equivalent stability properties as shown in~\cite[Theorem 6.3.1]{auriol2024contributions}.
\begin{proposition}\label{Lemma_equiv_stab}
    Exponential stability of~\eqref{sys-final} in the sense of Definition~\ref{def_exp_stab_IDE} is equivalent to exponential stability of~\eqref{eq:hyperbolic_couple} in the sense of Definition~\ref{Def_stability}.
\end{proposition}

\subsection{Stabilizing controller and robustness aspects}
In what follows, we assume that the open-loop system is not exponentially stable. {Note that the analysis presented here can be directly extended to systems where the control objective is to guarantee a decay rate that is not met by the open loop system, by a simple shift of the operator under consideration. For simplicity, only the results concerning the goal of exponential stabilization are presented here.}  The starting point of the controller design is a nominal law  for  (\ref{sys-final}) of the form
\begin{equation}\label{control-nominal}
u(t)=-\sum_{i=1}^k \sum_{j=1}^{\ell} H_{ij} x\left(t-(r_i+s_j)\right)-\int_{-\tau}^0 w_2(\theta) x(t+\theta)d\theta,
\end{equation}
which is based on canceling out the reflection terms {on the actuated boundary of the PDE} and leads to  \emph{target closed-loop dynamics} of retarded type, described by
\begin{equation}\label{sys-target}
x(t)=\int_{-\tau}^0 \left(w_1(\theta)-w_2(\theta) \right)x(t+\theta)d\theta
\end{equation}
and assumed exponentially stable.
The special case of $w_1=w_2$ in (\ref{control-nominal}) has been addressed in \cite{coron2017finite}, where fixed-time stability is obtained by compensating all dynamics.  Note that at the PDE level, this control law translates into imposing a homogeneous Dirichlet boundary condition for the leftward propagating states at $y=1$, see Section~\ref{parfixedtime}. 
However, the nominal controller (\ref{control-nominal}) may lack robustness. Moreover, as highlighted in~\cite{Auriol2019a,auriol2020robust,prevpaper}, canceling all reflection terms may result in \emph{zero delay-robustness margins}. %and poor tolerance to parameter uncertainties. 
Consequently, alternative strategies have been proposed such as partial reflection cancellation~\cite{auriol2018delay}, or the addition of a well-tuned low-pass filter~\cite{Auriol2023,prevpaper}.   We adopt the latter approach,  which is a natural solution since the fragility of stability stems from the non-strictly-proper nature of the control operator~\cite{wimfilter}, see also \cite{extraH,w-stable}. The inclusion of a first-order low-pass filter brings us to the control law
\begin{equation}
	T\dot u(t)+ u(t)=-\sum_{i=1}^k \sum_{j=1}^{\ell} H_{ij} x\left(t-(r_i+s_j)\right)
    -\int_{-\tau}^0 w_2(\theta) x(t+\theta)d\theta,\label{controlfinal}
\end{equation}
with $1/T$ the cutoff frequency of the filter.

\subsection{Objective and contributions}
In this paper we first derive necessary and sufficient conditions for which the direct application of (\ref{control-nominal}) leads to a fragile closed-loop system, in the sense of zero stability margins, and we show that the fragility issues are resolved by using control law (\ref{controlfinal}) with $T>0$ sufficiently small, where we rely on assessing the so-called filtered spectral abscissa~\cite{prevpaper}.   Our main focus is then on the qualitative and \emph{quantitative} characterizations of the choice of filter parameter  $T$, in relation to the allowable parametric uncertainty and delay margins.  The analysis highlights  the flexibility of the filter, namely high performance at low frequencies while guaranteeing robustness margins at high frequencies, thereby enabling an effective trade-off between performance and robustness.

\smallskip

The remainder of the article is structured as follows.
In Section~\ref{secstrong}, we adapt the notion of strong stability \cite{have:02,wimneutral2} to the structure of Equation (\ref{sys-final}) and we derive necessary and sufficient conditions that fully exploit the inter-dependence of the discrete delays. In Section~\ref{secfilter}, we show the stability of the closed-loop system (\ref{sys-final}) and (\ref{controlfinal}) for small $T>0$ and characterize the maximum stability interval in parameter $T$, leading to a frequency-sweeping test. In Section~\ref{secrobust}, on uncertainty quantification, we first derive sufficient conditions for closed-loop stability as a function of the combined PDE parameters, whose perturbations are translated into comparison system (\ref{controlfinal}), and the filter parameter $T$. Next, we address in depth a special class of problems, for which we characterize and compute stability crossing curves as a function of delay-mismatch and cut-off frequency of the filter.  The results uncover a distinctive, \emph{fractal-like structure} in the stability charts, an observation that, to the best of our knowledge, has not been previously reported in the literature. Finally, we state some concluding remarks in Section~\ref{secconcl}.

%%%%%%%%%%%%%%%%%%%%%%%%%%%%%%%%%%%%%%

\section{Strong stability condition of the open-loop system}\label{secstrong}

The zero solution of (\ref{sys-final}) with $u=0$ is exponentially stable if and only if the spectral abscissa
\[
c_0(\vec r,\vec s)\triangleq \sup_{\lambda\in\mathbb{C}}\left\{\Re(\lambda):\ \det \left(-I+\sum_{i=1}^k\sum_{j=1}^\ell H_{ij}e^{-\lambda(r_i+s_j)}\right. \right.
\left. \left.+\int_{-\tau}^0 w_1(\theta)e^{\lambda\theta}d\theta \right)=0  \right\}, 
\]
is negative, where we stress its dependence on delays $\vec r$ and $\vec s$ in the notation. It is well known that the exponential stability of delay-difference equations may be sensitive to \emph{infinitesimal} delay perturbations, which has led to the notion of strong stability in~\cite{have:02}.  Such fragility problems are invariably connected with the behavior of characteristic roots with large imaginary parts (see, e.g., \cite[Chapter 1]{bookwim}), which are in turn determined by the spectrum of the associated delay-difference equation
\begin{equation}\label{difference-open}
x(t)=\sum_{i=1}^k \sum_{j=1}^{\ell} H_{ij} x\left(t-(r_i+s_j)\right),
\end{equation}
as a consequence of the low-pass filtering property of the integral term in (\ref{sys-final}). Thus, only  perturbations of the discrete delays need to be considered in the analysis of fragility of stability. A particularity of (\ref{sys-final}) is that the $k\times \ell$ discrete delays do not correspond to independent parameters, but are linear combinations of $k+\ell$ parameters in the original PDE model. In what follows, we take this kind of dependence into account in the definition of strong stability.
\begin{definition}
The zero solution of (\ref{sys-final}) with $u=0$ is strongly  stable if there exist numbers $\epsilon_1>0$ and $\epsilon_2>0$ such that 
\[
c_0\left(\vec r+\Delta \vec r,\vec s+\Delta\vec s\right)<-\epsilon_2
\]
for all $\Delta\vec r\in \mathbb{R}^k$ and  $\Delta\vec s\in \mathbb{R}^\ell$ satisfying $\| \Delta\vec r\|_2<\epsilon_1$, $\| \Delta\vec s\|_2<\epsilon_1$, $r_i+\Delta r_i\geq 0,\ 1\leq i\leq k$, and 
 $s_j+\Delta s_j\geq 0,\ 1\leq j\leq \ell$.
\end{definition}
The specific  delay-dependency structure in (\ref{sys-final}) is within the scope of the analysis in \cite{wimneutral2}. Letting $\rho(\cdot )$ denote the spectral radius of a matrix, we can state the following strong stability condition.
\begin{proposition}
The zero solution of (\ref{sys-final}) with $u=0$ is strongly stable if and only if $c_0(\vec r,\vec s)<0$ and $\gamma_0<1$, with
\begin{equation}\label{defgamma0}
\gamma_0\triangleq %\max_{\vec\theta\in [0,\ 2\pi]^{k}} \max_{\vec\nu\in [0,\ 2\pi]^{\ell}}
\max_{\vec\theta\in [0,\ 2\pi]^{k},\ \vec\nu\in [0,\ 2\pi]^{\ell}}
  \rho \left( \sum_{i=1}^k \sum_{j=1}^{\ell} H_{ij} e^{\imath(\theta_i+ \nu_j) } \right).
\end{equation}
 \end{proposition}
\noindent\textbf{Proof.\ }  The correspondence in Equation (2.7) of \cite{wimneutral2} between the radius of the essential spectrum of the solution operator of a neutral equation with discrete delay and the spectral radius of the solution operator of the associated delay-difference equation carries over to the relation between the spectra of (\ref{sys-final}) and (\ref{difference-open}), which can be shown using Rouché's theorem.  The proof of the proposition then follows from a 
combination of  Proposition 3.10 and Theorem 4.3 in \cite{wimneutral2}.
\hfill $\Box$

\medskip

Note that the additional requirement in (\ref{defgamma0}) for strong stability (on top of a negative spectral abscissa) imposes strong stability of the delay-difference equation (\ref{difference-open}).
In contrast, the sufficient strong stability condition from the literature (see, e.g., \cite{Auriol2019a}) supplements a negative spectral abscissa with
\begin{equation}\label{stronglit}
\max 
\left\{  \rho \left( \sum_{i=1}^k \sum_{j=1}^{\ell} H_{ij} e^{\imath \theta_{ij}}  \right):\  \theta_{ij}\in[0,\ 2\pi] \right\}<1.
\end{equation}
It is more restrictive than condition $\gamma_0<1$ if $k\ell > k+\ell $, because expression (\ref{stronglit}) relies on treating the  delays in  every term of the right-hand side of (\ref{difference-open}) as independent variables.

\section{Filtered control law and closed-loop system}
\label{secfilter}

%Our starting point is an exponentially  stabilizing \emph{nominal} control law of the form
%\begin{equation}\label{control-nominal}
%u(t)=-\sum_{i=1}^k \sum_{j=1}^{\ell} H_{ij} x\left(t-%(r_i+s_j)\right)-\int_{-\tau}^0 w_2(\theta) x(t+\theta)d\theta,
%\end{equation}
%which is based on canceling out he reflection terms in (\ref{sys-%final}) and leads to  \emph{target closed-loop \,dynamics} of %retarded type, described by
%\begin{equation}\label{sys-target}
%x(t)=\int_{-\tau}^0 \left(w_1(\theta)-w_2(\theta) \right)x(t+%\theta)d\theta.
%\end{equation}

As a major complication, preventing a direct application of (\ref{control-nominal}), reflection terms cannot be safely canceled out, as already pointed out in \cite{have:02}. Small modeling errors on the delays, implementation errors, as well as communication and computation delays, will induce deviations between the values of $r_i$ and $s_i$ in the open-loop system, and estimates used in the control law. More precisely, if delay estimates $\hat r_i,\ i=1,\ldots,k$ and $\hat s_j,\ j=1,\ldots,\ell$ are used in the control law instead of the delays in the plant model, the delay-difference equation associated with the closed-loop system becomes
\[
x(t)=\sum_{i=1}^k \sum_{j=1}^{\ell} H_{ij} x\left(t-(r_i+s_j)\right)
- H_{ij} x\left(t-(\hat r_i+\hat s_j)\right).
\]
In the analysis of fragility of stability, the delays in the open-loop system and control law must be considered as independent parameters, even though their nominal values coincide. Hence, the target closed-loop system is strongly stable if and only if $c_{\mathrm{CL}}<0$, where 
\begin{align}
c_{\mathrm{CL}}\triangleq\sup_{\lambda\in\mathbb{C}}
\left\{\Re(\lambda): \det\left(-I+ \int_{-\tau}^0 \left(w_1(\theta)-w_2(\theta) \right)e^{\lambda\theta}d\theta\right)=0 \right\},\label{defCL}
\end{align}
is the spectral abscissa of (\ref{sys-target}), supplemented with the condition $\gamma_1<1$, with
\begin{align}\label{defgamma1}
\gamma_1\triangleq \max_{\scriptscriptstyle\begin{array}{l}\vec\theta\in [0,\ 2\pi]^{k}\\ \vec\nu\in [0,\ 2\pi]^{\ell}\end{array}}
\max_{\scriptscriptstyle\begin{array}{l}\vec\vartheta\in [0,\ 2\pi]^{k} \\ \vec\mu\in [0,\ 2\pi]^{\ell}\end{array}}
  \rho \left( \sum_{i=1}^k \sum_{j=1}^{\ell} H_{ij} e^{\imath(\theta_i+ \nu_j) }- H_{ij} e^{\imath(\vartheta_i+\mu_j)} \right).
\end{align}
It is important to point out that $\gamma_1$ is always an upper bound on  $\gamma_0$ and a strict upper bound if $\gamma_0\neq 0$.
\begin{proposition}
It holds that $\gamma_1\geq 2\gamma_0$.
\end{proposition}
\noindent\textbf{Proof.\ }
The minimal growth factor of two is already reached when restricting $\vec\vartheta$ and $\vec\mu$ to
\[
\begin{array}{l}
\vartheta_i= \left(\theta_i+\frac{\pi}{2}\right)\ \mathrm{mod} (2\pi),\ 1\leq i\leq k, \\
\mu_j= \left(\nu_j+\frac{\pi}{2}\right)\ \mathrm{mod} (2\pi),\ 1\leq j\leq \ell,
\end{array}
\]
in the maximization of the spectral radius in (\ref{defgamma1}). \hfill $\Box$

\medskip
 
In this article, we particularly target the situation where the open-loop system is unstable ($c_o\geq 0$) and
\begin{equation}\label{relgamma}
\gamma_0<1\leq \gamma_1, 
\end{equation}
that is, the delay-difference associated with the open-loop system is strongly stable, but directly canceling out the reflection terms by the control law leads to a fragile closed-loop system. A natural approach to avoid modifying the associated delay-difference equation (\ref{difference-open}) consists of adding a low-pass filter to the nominal control law. This leads us to the controller equation (\ref{controlfinal}).
%\begin{equation}\label{controlfinal}
%T\dot u(t)+ u(t)=-\sum_{i=1}^k \sum_{j=1}^{\ell} H_{ij} x\left(t-(r_i+s_j)\right)-\int_{-\tau}^0 w_2(\theta) x(t+\theta)d\theta,
%\end{equation}
%where $1/T$ is the cutoff frequency of the included first-order filter. 
%
\begin{remark}
In the case $\gamma_1<1$, the nominal, unfiltered control law already leads to a strongly stable closed-loop system provided that $c_{\mathrm{CL}}<0$, while in the case  $\gamma_0\geq 1$, strong stabilization is not possible.
\end{remark}

The characteristic equation of closed-loop system (\ref{sys-final}) and (\ref{controlfinal}) is given by 
$\det\Delta(\lambda;\ T)=0$,
with
\begin{multline}
\Delta(\lambda;\ T)\triangleq 
-I+\left(1-\frac{1}{1+\lambda T}\right)\sum_{i=1}^k\sum_{j=1}^\ell H_{ij}e^{-\lambda(r_i+s_j)}
+\int_{-\tau}^0 \left(w_1(\theta)-\frac{1}{1+\lambda T}w_2(\theta)\right)e^{\lambda\theta}d\theta \nonumber
\end{multline}
\subsection{Preservation of stability by the filter's inclusion}

The fragility problems related to delay perturbations, addressed above, are resolved by filtering the control law as in (\ref{controlfinal}), on the condition that the filter itself is not destabilizing. Even though stability for small values of $T$ cannot be guaranteed in general (see the counter examples in \cite{wimfilter,prevpaper}), the approach is effective if condition (\ref{relgamma}) is satisfied, which is a consequence of  the following theorem.

\begin{theorem} \label{thmmargin} There exists a number $\hat T>0$ such that the zero solution of (\ref{sys-final}) and (\ref{controlfinal}) is exponentially stable for all $T\in(0,\ \hat T)$ if and only if $c_{\mathrm{CL}}<0$ and $\gamma_0<1$, with  $c_{\mathrm{CL}}$ and $\gamma_0$ defined by (\ref{defCL}) and (\ref{defgamma0}), respectively.
\end{theorem}
\noindent\textbf{Proof.\ } According to Corollary 5.5 of  \cite{prevpaper}, the assertion of the theorem holds if and only if $c_{\mathrm{CL}}<0$  and the so-called filtered spectral abscissa $c_F$ of (\ref{difference-open}) satisfies $c_F<0$.

 Theorem 3.12 of \cite{prevpaper}, applied to (\ref{difference-open}), stipulates that $c_F$ is the unique zero of the (strictly) decreasing function $f:\ \mathbb{R}\to\mathbb{R}$, defined by
\[
f(r)\triangleq 
\max_{\scriptsize \begin{array}{c}
     \vec\theta\in [0,\ 2\pi]^{k}, \\
     \vec\nu\in [0,\ 2\pi]^{\ell}
\end{array}}
  \alpha \left( -I+\sum_{i=1}^k \sum_{j=1}^{\ell} H_{ij} e^{-r(r_i+s_i)} e^{\imath(\theta_i+ \nu_j) } \right),
\]
with $\alpha(\cdot)$ the spectral abscissa of a matrix.
Since $f$ is decreasing, $c_F<0$ is equivalent to $f(0)<0$, which on its turn is equivalent to $\gamma_0<1$.
\hfill $\Box$

\subsection{Characterization of maximum $T$ to preserve stability}

  In this section, we determine the largest stability interval  of (\ref{sys-final}) and (\ref{controlfinal}) in parameter $T$ that contains $T=0$, under the conditions of Theorem~\ref{thmmargin}. If the interval is finite, we denote by $\tilde T$ the right end point, otherwise we set $\tilde T=+\infty$.

For $\lambda\neq 0$, the characteristic equation of the closed-loop system formed by (\ref{sys-final}) and (\ref{controlfinal}), namely $\Delta(\lambda;\ T)=0$, can be rephrased in the form
\begin{multline}\label{critdel}
\frac{1}{\lambda T}\in\sigma\left(  
\left(I+\int_{-\tau}^0 (w_2(\theta)-w_1(\theta))e^{\lambda\theta}d\theta   \right)^{-1}
\left(-I+\sum_{i=1}^k\sum_{j=1}^\ell H_{ij}e^{-\lambda(r_i+s_j)}
+\int_{-\tau}^0 w_1(\theta)e^{\lambda\theta}d\theta 
 \right)
\right),
\end{multline}
with $\sigma(\cdot)$ denoting the spectrum.
Note that $\lambda=0$ cannot be a characteristic root since $\Delta(0; T)=0$ for any value of $T$ implies $\Delta(0; T)=0$ for all $T\geq 0$, which contradicts the exponential stability of the target system, implied by the condition $c_{\mathrm{CL}}<0$.
Critical values of $T$, where a stability switch occurs, are connected with the presence of characteristic roots on the imaginary axis. Substituting $\lambda=\imath\omega,\ \omega>0$,  brings us to the stu\,dy of the function $\omega\mapsto G(\omega)$, with
\begin{multline}
G(\omega)\triangleq \left(I+\int_{-\tau}^0 (w_2(\theta)-w_1(\theta))e^{\imath\omega\theta}d\theta   \right)^{-1}
\left(-I+\sum_{i=1}^k\sum_{j=1}^\ell H_{ij}e^{-\imath\omega(r_i+s_j)}
+\int_{-\tau}^0 w_1(\theta)e^{\imath\omega\theta}d\theta 
\right).    \nonumber
\end{multline}
 Since $\frac{1}{\imath\omega T}$ is always on the imaginary axis, we are interested in values of $\omega$ for which $G(\omega)$ has an imaginary axis eigenvalue. 
\begin{proposition}\label{propfinom}
There are only finitely many values of $\omega$ for which $G(\omega)$ has an eigenvalue on the imaginary axis.	
\end{proposition}	
\noindent\textbf{Proof.}\
For large $\omega$, the function $\omega \mapsto G(\omega)$ converges to the function $\omega\mapsto G_{\infty}(\omega)$, with
\[
G_{\infty}(\omega)\triangleq -I+\sum_{i=1}^k\sum_{j=1}^\ell H_{ij}e^{-\imath\omega(r_i+s_j)},
\]
in the sense that for all $\epsilon>0$ there is an $\underline{\omega}>0$ such that $\|G(\omega)-G_{\infty}(\omega) \|_2<\epsilon$ for all $\omega\geq \underline{\omega}$.
As argued in the proof of Theorem~\ref{thmmargin}, it holds that  $f(0)<0$, which implies that the set $\bigcup_{\omega\geq 0}  \sigma(G_{\infty}(\omega))$ is confined to the open left half plane and bounded away from the imaginary axis.\hfill $\Box$

\medskip

The form of (\ref{critdel}) and Proposition~\ref{propfinom} suggest Algorithm~\ref{algsweep} for computing $\tilde T$.  It will be illustrated in Section~\ref{secfilter}-C.
\begin{algorithm}
{\small 
\begin{algorithmic}[1]
\State Determine the complete set $\{(\omega_1,\gamma_1),\ldots,(\omega_l,\gamma_l) \}\subseteq\mathbb{R}_{\geq 0}\times \mathbb{R}$ such that $G(\imath\omega_{i})$ has eigenvalue $\imath\gamma_i,\ i=1,\ldots,\ell$, by frequency sweeping
\If{the set is empty}
 \State $\tilde{T}=+\infty$, 
 \Else 
 \State $\tilde{T}=\min_{i\in\{1,\ldots,l\}} T_i$,\ \
	with  $T_i=\frac{1}{\omega_i|\gamma_i|},\ i=1,\ldots,l$
    \EndIf
    \end{algorithmic}}
    \caption{\label{algsweep} Computation of $\tilde{T}$}
\end{algorithm}
%
%Its application to the example in the previous section leads to $\ell=3$  and $T_1\approx  0.3056,\ T_2\approx0.3806$ and $T_3\approx0.1615$, from which we conclude $\tilde{T}=T_3$.  These values are indicated in Figure~\ref{figstabdist}.

% \paragraph{Analysis of a special case}
% Let us assume that $m=n=1$ and that $w_1$ is a constant function. We also consider that $w_2=w_1$.
% We need to analyze the following equation:
% \begin{align}
%     1+T\imath \omega-T(H_1\imath \omega\mathrm{e}^{\imath \tau  \omega}-w_1\mathrm{e}^{\imath \tau  \omega}+w_1)=0.
% \end{align}
% This leads to 
% \begin{align}
%     &1-TH_1\omega\sin(\tau \omega)+Tw_1\cos(\tau \omega)-Tw_1=0, \label{eq_1_T}\\
%     &T\omega-TH_1\omega \cos(\tau \omega)-Tw_1\sin (\tau \omega)=0.
% \end{align}
% The second equation leads to 
% \begin{align}
%     \omega-H_1\omega \cos(\tau \omega)-w_1\sin (\tau \omega)=0 \label{eq_omega}
%     \end{align}
%     Equation~\eqref{eq_omega} has a finite number of solution, that we denote $\omega_c$. We can then  $T_0^\star$ as 
%     \begin{align}
%         T_0^\star=\min_{\omega_c}(\frac{1}{H_1\omega_c\sin(\tau\omega_c)+w_1-w_1\cos(\tau \omega_c)},~\text{s.t. $H_1\omega_c\sin(\tau\omega_c)+w_1-w_1\cos(\tau \omega_c)>0$}).
%     \end{align}
%     If $T_0^\star$ is not defined, then  $T^\star=+\infty$, otherwise $T^\star=T_0^\star tsc$.
% % \begin{align*}
% %     (1-Tw_1)^2 +T\omega^2= T^2H_1^2\omega^2+T^2w_1^2-4w_1H_1\omega T^2
% % \end{align*}

\section{Robustness against parametric perturbations}
\label{secrobust}
The aim of the section to study the choice of $T$ in relation to robustness against parametric uncertainty, assuming condition (\ref{relgamma}) is satisfied and target system (\ref{sys-target}) is exponentially stable.  In the extreme case of $T=0$, the stability margin against delay perturbations is zero. This is not in contrast with the property that the characteristic function of the closed-loop system converges on compact sets to the one of the exponentially stable target system (\ref{sys-target}) as $T\rightarrow 0+$, because the destabilizing perturbations may induce characteristic roots crossings of the imaginary axis with arbitrarily large imaginary part. While with the inclusion of a filter, the fragility problem  is theoretically resolved (Theorem~\ref{thmmargin}), it is expected that the smaller the value of $T$ is (the higher the cut-off frequency), the smaller the delay margin is, which we will confirm by theoretical results and  a numerical case-study. On the other hand, large values of $T$ may induce the low-frequency dynamics to be significantly different from those of the desired closed-loop system. This may even be another source of instability, as we shall illustrate with the application of Algorithm~\ref{algsweep}, leading to an inherent design trade-off in the choice of~$T$.

 \subsection{Description of the model mismatch}
 %\textcolor{red}{Proposed changes in references to Assumption 1 in the paper are all marked in red, to double-check...}

We start by describing the considered mismatch in terms of a {ball-like set} in the space of PDE parameters.
\begin{definition}%[$\mathcal{B}_{\mathcal{P}}^{\delta}$]
\label{def:ball}
    Given parameters $\mathcal{P} \triangleq(\Lambda,\Sigma,Q,R)$ for system \eqref{eq:hyperbolic_couple} satisfying \eqref{eq:original_order}, and a positive scalar $\delta>0$, we define the set of \emph{valid $\delta$-perturbations of $\mathcal{P}$}, as the set $\mathcal{B}_{\mathcal{P}}^{\delta}$ of all parameter tuples $(\bar{\Lambda},\bar{\Sigma},\bar{Q},\bar{R})$, such that the following properties hold:
    \begin{itemize} 
        \item[(i) ]$\bar{\Lambda}=\text{diag}\big(\bar{\lambda}_1,\ldots, \bar{\lambda}_{k},-\bar{\mu}_1,\ldots, -\bar{\mu}_\ell\big)$, with
        \begin{equation}\label{order}
            -\bar{\mu}_\ell<\cdots<-\bar{\mu}_1<0<\bar{\lambda}_1<\cdots<\bar{\lambda}_{k}
        \end{equation}
        \item[(ii) ] The coefficients satisfy \begin{align} \label{eq:R_delta}
        |\bar{R}_{ij}-R_{ij}|&\leq \delta|R_{ij}|, \quad \text{$1 \leq i\leq \ell$, $1 \leq j\leq k,$}\\
        |\bar{Q}_{ij}-Q_{ij}|&\leq \delta|Q_{ij}|, \quad \text{$1 \leq i\leq k$, $1 \leq j\leq \ell,$} \\
        |\bar{\Sigma}_{ij}-\Sigma_{ij}|&\leq \delta|\Sigma_{ij}|, \quad \text{$1 \leq i,j\leq k+\ell$,} \\ \label{eq:Lambda_delta}
        |\bar{\Lambda}_{ii}-\Lambda_{ii}|&\leq \delta|\Lambda_{ii}|, \quad \text{$1\leq i \leq k+\ell$}. 
        \end{align}
    \end{itemize}
\end{definition}

\begin{proposition}\label{prop:closed_ball}
    %Given parameters $\mathcal{P}=(\Lambda,\Sigma,Q,R)$ for system \eqref{eq:hyperbolic_couple}, there 
    There exists a maximal real constant $\delta^{\star}(\mathcal{P})>0$, such that, for all $\delta \in(0,\delta^{\star})$, the set $\mathcal{B}_{\mathcal{P}}^{\delta}$ is closed.
\end{proposition}
\noindent\textbf{Proof.}\
    The proof of this result is straightforward, since \eqref{eq:Lambda_delta} together with \eqref{eq:original_order} imply \eqref{order} for sufficiently small values of $\delta$ and \eqref{eq:R_delta}-\eqref{eq:Lambda_delta} define a closed set. Furthermore, the (finite) maximal constant $\delta^{\star}(\mathcal{P})$ is directly related to the minimal gap between consecutive values in \eqref{eq:original_order}. \hfill $\Box$

    \smallskip

%In the remainder of the article, we will assume, when considering parametric uncertainties in the coefficients of \eqref{eq:hyperbolic_couple} around a nominal set of coefficients $\mathcal{P}$.

%In what follows we assess the impact of model mismatch, expressed by perturbed parameters $\mathcal{\bar P}\triangleq(\bar\Lambda,\bar\Sigma,\bar Q,\bar R)$ on exponential staiblity, where we assume for  
%
Inherently to the adopted delay-based approach, we need to translate perturbed PDE parameters $\mathcal{\bar P}\triangleq(\bar\Lambda,\bar\Sigma,\bar Q,\bar R)$, assuming relation (\ref{order}) is respected, into the matrices, discrete delays and delay kernel $w_1$ in IDE (\ref{sys-final}), which are perturbed to $\bar{r}_i$, $\bar{s}_j$, $\bar{\tau}$ and $\bar{w}_1$, respectively.  Note that the function  $\bar{w}_1$ may not be defined over the same interval as $w_1$ due to the mismatch on the velocity matrix. Letting 
  $\mathds{1}_{\Omega}$ denote the indicator function for the set $\Omega$, we resolve this issue by introducing the extended kernels $w_1^{e}$ and $\bar{w}_1^{e}$, 
defined for all $
   \theta\in [-\max(\tau,\bar{\tau}),0]$
by
\begin{equation}
    w_1^{e}(\theta)=w_1(\theta)\,\mathds{1}_{[-\tau,0]}(\theta)
\end{equation}
and
\begin{equation}
    \bar{w}_1^{e}(\theta)=\bar{w}_1(\theta)\,\mathds{1}_{[-\bar{\tau},0]}(\theta),
\end{equation}
so that both kernels are consistently defined on the same domain.  The following lemma, whose proof can be found in Appendix~\ref{secapB}, relates the size of the perturbations at PDE level and  at IDE level.

\begin{lemma} \label{lemma:convergence_w} {Given parameters $\mathcal{P}$ for system \eqref{eq:hyperbolic_couple}, l}et $\hat\delta\in(0,\delta^\star)$, with $\delta^\star$ determined by Proposition~\ref{prop:closed_ball}.  There exists a constant $C_w>0$ such that for any $\delta\in[0,\hat \delta]$  and $\mathcal{\bar P}\in\mathcal{B}_{\mathcal{P}}^\delta$, it holds that
\begin{multline}\label{eq:ineg_w}
|(\bar{w}_1^{e})_{ij}(\theta)-(w_1^{e})_{ij}(\theta)|_2 \leq  C_w \big(\delta+\mathds{1}_{\bar{D}^w}(\theta)\big),\
\theta \in [-\max(\tau,\bar{\tau}),0],\  1\leq i,j \leq \ell,
\end{multline}
%and a function $(0,\hat\delta)\ni\delta\mapsto \bar D^w(\delta)\in[0,\tau]$  satisfying $\nu(\bar{D}^w(\delta))\leq C_w\delta$, with $\nu$ the Lebesgue measure
where  $\bar D^w\in[0,\tau]$  satisfies $\nu(\bar{D}^w)\leq C_w\delta$ and $\nu$ is the Lebesgue measure.
\end{lemma}
As can be observed, the right-hand side of~\eqref{eq:ineg_w} consists of two terms. More precisely, the difference between the kernels $w_1^{e}$ and $\bar{w}_1^{e}$ is uniformly small (bounded by $C_w \delta$), except on a subset whose Lebesgue measure vanishes as $\delta \to 0$. This property follows from the piecewise continuity of the kernel $w_1^{e}$, and is made explicit in the proof of the lemma.
Considering now also the induced change of parameters $H_{ij}$ to $\bar{H}_{ij}$,\ $r_i$ to $\bar{r}_i$ and $s_j$ to $\bar{s}_j$, $\ 1\leq i\leq k,\ 1\leq j\leq\ell$, we can state the following result.
\begin{proposition} \label{lemeps}    Let $\hat\delta\in(0,\delta^\star)$ {with $\delta^\star$ defined by Proposition~\ref{prop:closed_ball}}.
There exists  a number $\kappa>0$ such that
for any $\delta \in[0,\hat\delta]$  and $\mathcal{\bar P}\in\mathcal{B}_{\mathcal{P}}^\delta$, we have
\[
\begin{array}{l}
|\bar{r}_i-r_i|\leq \kappa\delta,\  
\
|\bar{s}_j-s_j|\leq \kappa \delta,
\\
\| \bar{H}_{ij}-H_{ij}\|_2\leq \kappa \delta,\ \hspace*{2cm}
\\
\left\| \int_{-\tau}^0 w_1(\theta)e^{\lambda\theta}d\theta-\int_{-\bar{\tau}}^0 \bar{w}_1(\theta)e^{\lambda\theta} d\theta\right\|_{\mathcal{H}_{\infty}}<\kappa\delta,  
\end{array}
\]
for   $1\leq i\leq k,\ 1\leq j\leq\ell$.
\end{proposition}
\noindent\textbf{Proof.\ } 
The two first inequalities can be directly obtained from the definition of the parameters $r_i$, and $s_j$. Similarly, as the matrices $H_{ij}$ are defined through the matrices $R$ and $Q$ (see~\cite{Auriol2019a}), we immediately get the third inequality. The last inequality is a direct consequence of~\eqref{eq:ineg_w}.%\textcolor{blue}{TBC}
\hfill $\Box$

%More precisely, we consider model mismatch in the form of constant perturbations on the parameters of the PDE model and assume the control law is based on parameters
%\[
%\begin{array}{l}
%\lambda_i +\delta\lambda_i,\ 1\leq i\leq k, \\
%\mu_j+\delta\mu_j,\ j\leq i\leq \ell, \\
%\end{array}
%\]
%and matrices
%\[
%\Sigma+\delta \Sigma,\ Q+\delta Q, R+\delta R.\ 
%\]

\subsection{Qualitative description of the stability region}

We analyze the effects of the model mismatch in the controller design by characterizing stability of the closed-loop system in the $(T,\delta)$-parameter space. Letting the \emph{shape} of the desired kernel of the distributed delay in the target closed-loop system be described by the piecewise continuously differentiable function $w_{\mathrm{des}}:\ [-1,\ 0]\to \mathbb{R}^{\ell\times \ell}$,  the filtered control law based on \emph{assumed }model parameters $\mathcal{\bar P}\triangleq(\bar\Lambda,\bar\Sigma,\bar Q,\bar R)$ takes the form
\begin{multline}\label{controlfinalper}
T\dot u(t)+ u(t)=-\sum_{i=1}^k \sum_{j=1}^{\ell}  \bar{H}_{ij}\ x\left(t-(\bar{r}_i+\bar{s}_j)\right)
-\int_{-\bar{\tau}}^0  \left( \bar{w}_1(\theta) - \frac{1}{\bar{\tau}}w_{\mathrm{des}}\left(\frac{\theta}{\bar{\tau}}\right)  \right)x(t+\theta)d\theta.
\end{multline}
The following theorem addresses stability robustness of the closed-loop system.
\begin{theorem}\label{theo:cone}
Assume that the closed-loop system  (\ref{sys-final}) and (\ref{controlfinal}), with $w_2(\theta)=w_1(\theta)-\frac{1}{\tau} w_{\mathrm{des}}\left(\frac{\theta}{\tau}\right),\ \theta\in[-\tau,\ 0]$,  is strongly  stable.	Then there exist a constant $\hat T>0$  and  a constant $c>0$  such that  for all $(T,\delta)$ satisfying
	\begin{equation}\label{spie}
	T\in(0,\hat T) ,\ 0<\delta < c T,
	\end{equation}
    and for all  $\mathcal{\bar P}\in \mathcal{B}_{\mathcal{P}}^{\delta}$, the closed-loop system  (\ref{sys-final}) and (\ref{controlfinalper}) is exponentially stable 
    %with $\mathcal{B}_{\mathcal{P}}^{\delta}$ as in Definition~\ref{def:ball}.
\end{theorem}
	 
\smallskip
\noindent\textbf{Proof.\ }
After application of the stability-preserving operator $\left(\frac{d}{dt}+ T I\right)$, the closed-loop equation of  (\ref{sys-final}) and (\ref{controlfinalper}),  can be decomposed as another feedback interconnection of the ``nominal model''
{\small\begin{multline}
\left\{\begin{array}{l}
T\left(\dot x(t)-\sum_{i=1}^k \sum_{j=1}^{\ell} H_{ij} \dot x\left(t-(r_i+s_j)\right)\right)=-x(t)
\\
%-\int_{-\tau}^0 w_2(\theta)x(t+\theta)d\theta+
+T \int_{-\tau}^0 w_1(\theta) \dot x(t+\theta)d\theta+\int_{-\tau}^0 \frac{1}{\tau} w_{\mathrm{des}}\left(\frac{\theta}{\tau}\right) x(t+\theta)d\theta+\\
\sum_{i=1}^k \sum_{j=1}^{\ell} (H_{ij} u_{ij}(t)+\tilde u_{ij}(t))+ u_0(t), 
\\
y_{ij}(t)=\dot x(t),\ \ 1\leq i\leq k,\ 1\leq j\leq\ell, \\
\tilde y_{ij}(t)=-x(t-(r_i+s_j))+u_{ij}(t),\ \ 1\leq i\leq k,\ 1\leq j\leq\ell, \\
y_0(t)=x(t),
\end{array}\right. \label{nominalmodel}
\end{multline}}
coupled with an uncertainty block described by
\begin{equation}\label{uncertainty}
\left\{ \begin{array}{rcl}
u_{ij}(t)&=& \int_{t- \bar{r}_i -\bar{s}_j}^{t-r_i-s_j} y_{ij}(\theta)d\theta, \ \ 1\leq i\leq k,\ 1\leq j\leq\ell, \\
\tilde u_{ij}(t)&=&  \left(\bar{H}_{ij}-H_{ij}\right) \tilde y_{ij}, \ \ 1\leq i\leq k,\ 1\leq j\leq\ell,\\
u_0(t)&=&\int_{-\tau}^0 w_1(\theta)y_0(t+\theta)d\theta-\int_{-\bar{\tau}}^0  \bar{w}_1(\theta) y_0(t+\theta)d\theta
\\
&& +\int_{-\bar{\tau}}^0 \frac{1}{\bar{\tau}} w_{\mathrm{des}}\left(\frac{\theta}{\bar{\tau}}\right) y_0(t+\theta)d\theta
\\
&&-\int_{-\tau}^0 \frac{1}{\tau} w_{\mathrm{des}}\left(\frac{\theta}{\tau}\right) y_0(t+\theta)d\theta,
\end{array}\right.
\end{equation}
We denote by $\mathcal{G(\lambda)}$ and $\Delta(\lambda)$ the transfer function of (\ref{nominalmodel}) and (\ref{uncertainty}), respectively.

Instrumental to the analysis of $\mathcal{G}$,  we first consider the transfer function
\begin{align*}
&G(\lambda)\triangleq \left[T\lambda\left(I-\sum_{i=1}^k\sum_{j=1}^{\ell} H_{ij} e^{-\lambda(r_i+s_j)} \right.\right. \nonumber\\
&\left.\left.-\int_{-\tau}^0 w_1(\theta) e^{\lambda\theta} d\theta \right) \right.\left.+I- \int_{-\tau}^0 \frac{1}{\tau} w_{\mathrm{des}}\left(\frac{\theta}{\tau}\right) x(t+\theta)d\theta\right]^{-1}.
\end{align*}
We note that
\[
\det\left(I - \int_{-\tau}^0 \frac{1}{\tau} w_{\mathrm{des}}\left(\frac{\theta}{\tau}\right) x(t+\theta)d\theta\right),
\]
is bounded away from zero on the imaginary axis, since the target system is assumed exponentially stable, while for small values of $T\lambda,\ \lambda\in\imath\mathbb{R}$, the first term of the inverted matrix acts as a bounded perturbation.  Using this observation and the expression of an inverse using the adjugate matrix, we conclude there exist numbers $R>0$ and $M_1>0$ such that
\begin{equation} \label{infnorm1}
\|\imath \omega\, G(\imath\omega)\|_2\leq \ell_1\omega,\ \ \ 0\leq \omega T\leq R.
\end{equation}
Letting $\tilde\omega=\omega T$, we can also express
\begin{multline}
 G(\imath\omega)=\left[ \left( I-\sum_{i=1}^k\sum_{j=1}^{\ell} H_{ij} e^{-\imath\tilde \omega\frac{r_i+s_j}{T}}
-\int_{-\tau}^0 w_1(\theta)e^{\imath \frac{\tilde\omega}{T}\theta}d\theta  \right)  \right.
\\
\left. +  \frac{I}{\imath \tilde \omega} 
- \frac{1}{\imath \tilde \omega} \int_{-\tau}^0 (w_1(\theta)-w_2(\theta)) e^{\imath \frac{\tilde\omega}{T} \theta}d\theta      \right]^{-1}\frac{1}{\imath\omega\,T} .
\end{multline}
We now consider the case where $\omega T>R$ or, equivalently, $\tilde\omega> R$.
If we let $T\rightarrow 0+$, the integral terms converge to zero uniformly for $\tilde \omega\in [R,\ +\infty)$.  Furthermore, the determinant of
\[
I-\sum_{i=1}^k\sum_{j=1}^{\ell} H_{ij} e^{-\imath\tilde \omega\frac{r_i+s_j}{T}},
\]
is bounded away from zero for $\tilde \omega\in\mathbb{R}$, following from the condition $\gamma_0<1$, and this remains so when adding the term 
$\frac{1}{\imath \tilde \omega}$, as it only induces shifts parallel to the imaginary axis. Consequently, there exist numbers $M_2>0$ and $\tilde T>0$ such that
\begin{equation}\label{infnorm2}
\|\imath\omega\, G(\imath\omega)\|_2\leq \frac{M_2}{T},\ \ \ \omega T\geq R,\ T<\tilde T.
\end{equation}
 From (\ref{infnorm1}) and (\ref{infnorm2}) we conclude
 \begin{equation}\label{hinfc1}
 \|\lambda G(\lambda)\|_{\mathcal{H}_{\infty}}\leq  \frac{\max(M_1 R,M_2)}{T},\ \ \forall T< \tilde T.
 \end{equation}
Using the same kind of arguments and partitioning of $\omega$-intervals, one can show that there is a constant $M_3$ such that 
 \begin{equation}\label{hinfc2}
 \|G(\lambda)\|_{\mathcal{H}_{\infty}}\leq  M_3.
 \end{equation}
From (\ref{hinfc1}) and (\ref{hinfc2}) it follows that there is a $\tilde  T>0$ and $M>0$ such that
\[
\|\mathcal{G}(\lambda)\|_{\mathcal{H}_{\infty}}\leq  \frac{M}{T},\ \ \forall T<\tilde T.
\]

We now turn our attention to the uncertainty block.   We have
\[
\begin{array}{l}
\left\| \int_{-\bar{\tau}}^0 \frac{1}{\bar{\tau}} w_{\mathrm{des}}\left(\frac{\theta}{\bar{\tau}}\right) e^{\lambda\theta}d\theta
-\int_{-\tau}^0 \frac{1}{\tau} w_{\mathrm{des}}\left(\frac{\theta}{\tau}\right) e^{\lambda\theta}d\theta\right\|_{\mathcal{H}_{\infty}}
\\
=\left\| \int_{-1}^0 w_{\mathrm{des}}\left(\theta\right) \left(e^{\lambda \bar{\tau}\theta}-e^{\lambda\tau\theta}\right)\right\|_{\mathcal{H}_{\infty}},
\end{array}
\]
which has an upper bound proportional to $|\tau-\bar\tau|$. Combining this observation with
Proposition~\ref{lemeps} and the structure of (\ref{uncertainty}), we conclude that there exist constants $\bar\kappa>0$ and $\bar\delta\in(0,\delta^\star)$ such that 
\[
\|\Delta(\lambda)\|_{\mathcal{H}_{\infty}}\leq \bar\kappa \delta,
\]
if $\delta<\bar\delta$.

Finally, the closed-loop system is stable if the product of the induced gains of nominal system and uncertainty block is smaller than one, leading to the conditions
\[
 \bar\kappa M \, \delta < T,\ \ \ \  T< \tilde T,\ \  \delta<\bar\delta. 
\]
Setting $c=1/(\bar\kappa M)$ and subsequently $\hat T=\min(\tilde T, \bar\delta/c)$, the proof is completed.  \hfill $\Box$

\medskip

\begin{remark} If a complementary point of view is taken based on model mismatch, where the nominal control law is considered fixed and the plant parameters deviate from the nominal values, the assertion of Theorem~\ref{theo:cone} remains valid. The main difference in the proof is that no uncertainty term describing the impact of reshaping function $w_{\mathrm{des}}$ needs to be considered.
\end{remark}

Note that condition (\ref{spie}) implies that locally a \emph{cone-like structure} is included in the stability region in the space of PDE parameters and filter parameter $T$. Even though (\ref{spie})  is based on a sufficient stability condition,  such a structure is effectively observed in stability charts under condition (\ref{relgamma}), which implies that without filter the closed-loop system is not strongly stable. This is illustrated in the following Section~\ref{seccase} with a case-study.

%We can remark
%\begin{enumerate}
%\item From the proof it follow that for small $T$ transition to instability (by changing PDE parameters) is always related to high-frequency characteristic roots.	
%\item Reading may wonder whether cone is related to conservatism $\rightarrow$ case studies %later confirm it is actually a cone
%\end{enumerate}

\subsection{Stability charts for a special case}\label{seccase}

We consider the PDE system (\ref{eq:hyperbolic_couple}) described by scalars $Q,\ R$ and matrices
\begin{align*}
    \Lambda=\begin{pmatrix}
        \lambda_1 &0 \\ 0 & -\mu_1
    \end{pmatrix}, \quad \Sigma=\begin{pmatrix}0 & \Sigma^{+-}\\ \Sigma^{-+}&0\end{pmatrix}.
\end{align*}
 Then Equation (\ref{sys-final}) takes the form
\begin{equation}\label{plantpde}
x(t)= H_{11} x(t-\tau)+\int_{-\tau}^0 w_1(\theta) x(t+\theta)\; d\theta+ u(t), 
\end{equation}
with $\tau=\frac{1}{\mu_1}+\frac{1}{\lambda_1}$, $H_{11}=RQ$, and where the function $w_1$ is given by
\begin{align}
	w_1(\theta) =& \left(\frac{a}{\tau} + \frac{d^{-}(\theta)}{\tau^2}\right) J_0\left(2 \sqrt{h^{-}(\theta)}\right) \nonumber \\
    &+ \frac{d^{-}(\theta)}{\tau^2} J_2\left(2 \sqrt{h^{-}(\theta)}\right), \quad \theta \in [-\tau, 0],\label{eq:explicit-N-Bessel}
\end{align}
in which
\[
\begin{aligned}
	h^{-}(\theta) & \triangleq -\frac{\Sigma^{+-} \Sigma^{-+}}{\lambda_1 \mu_1  \tau^2}\theta(\tau+\theta), &  d^{-}(\theta) & \triangleq \frac{\Sigma^{+-} \Sigma^{-+}}{\lambda_1 \mu_1}(\tau+\theta b), \\
	a & \triangleq \frac{Q}{\mu_1}\Sigma^{-+} + \frac{R}{\lambda_1}\Sigma^{+-},  & b & \triangleq 1 + QR.
\end{aligned} 
\]
and the functions $J_i,\ i\geq 0$, denote the Bessel functions of the first kind.
For more details, we refer to~\cite{saba2019stability}.

We assume specific model mismatch described by one parameter $\varepsilon\in(-1,\ \infty)$,   perturbing $(\lambda_1,\mu_1,\Sigma^{+-},\Sigma^{-+} )$ to
\begin{equation}\label{parerex}
\left(\frac{\lambda_1}{1+\varepsilon}, \frac{\mu_1}{1+\varepsilon},\frac{\Sigma^{+-}}{1+\varepsilon},  \frac{\Sigma^{-+}}{1+\varepsilon} \right),
\end{equation}
which has a clear physical implication, namely that the ``round-trip" time $\tau$ is subject to relative estimation error $\varepsilon$, while the net increase/decay of a signal along a cycle is correctly estimated. From (\ref{eq:explicit-N-Bessel}) it can be seen that this mismatch leads  to a rescaling of the kernel $w_1$, while the value of $H_{11}$ is not altered. Now, the control law (\ref{controlfinalper}) simplifies to
%We consider the following controller for (\ref{plantpde}),
\begin{multline} \label{controlpde}
T\dot u(t)+u(t)=-H_{11} x(t-(1+\varepsilon)\tau)- \int_{-\tau(1+\varepsilon)}^0  x(t+\theta)
\\
\times \left(\frac{1}{1+\varepsilon} w_1\left(\frac{\theta}{1+\varepsilon}\right) -\frac{1}{(1+\varepsilon)\tau}w_{\mathrm{des}}\left(\frac{\theta}{(1+\varepsilon)\tau}\right)  \right)\,d\theta,
\end{multline}
parametrized by the pair $(T,\varepsilon)$.   %\textcolor{red}
{Note that the parametric perturbations (\ref{parerex}) 
%satisfy {Assumption~\ref{Assum:Bound_mismatch}} for small values of $\varepsilon$, with $\delta=|\varepsilon|$.
imply $\mathcal{\bar P}\in\mathcal{B}_{\mathcal{P}}^{2\epsilon} $ for small values of $\varepsilon$.
%, with $\delta=|\varepsilon|$.
}

In the numerical experiments reported in what follows, we take the  numerical values 
\begin{equation}\label{numvalpde}
\lambda_1=2, \ \mu_1=\frac{\sqrt{2}}{2},\ \Sigma^{+-}=5,\ \Sigma^{-+}=1,\ Q=1.2,\ R=0.5,
\end{equation}
which lead to $\tau=\frac{1}{2}+\sqrt{2}$, $H=0.6$. For the controller,  $w_{\mathrm{des}}$ is specified as the constant function $w_{\mathrm{des}}=0.45$
resulting in the target \,dynamics
\[
x(t)=\int_{-\tau}^{0} \frac{0.45}{\tau} x(t+\theta)d\theta.
\]
Then the open-loop system is unstable and the target system exponentially stable, with spectral abscissa equal to $1.735$ and $-0.7468$, respectively. Furthermore, we have
\[
\gamma_0=|H_{11}|<1,\ \gamma_1=2|H_{11}|>1,
\]
and from Theorem~\ref{thmmargin}, the filtered closed-loop system is exponentially stable for small values of $T$ if $\varepsilon=0$. To characterize the stability margin in $T$, an application of Algorithm~\ref{algsweep} leads us to $l=3$  and $T_1\approx  0.3056,\ T_2\approx0.3806$ and $T_3\approx0.1615$, from which we conclude $\tilde{T}=T_3$.  These values are indicated in Figure~\ref{figstabdist}.

\subsubsection{Behavior for small $T$ and $\varepsilon$} \label{parsmallT}

We first construct the stability chart of the closed-loop system  (\ref{plantpde}) and (\ref{controlpde}) in the $(T,\varepsilon)$-parameter  space   where we ignore the two distributed delay terms, which brings us to 
\[
x(t)=H_{11}x(t-\tau)+u(t),\ \  T\dot u(t)+u(t)=-H_{11}x(t-(1+\varepsilon)\tau).
\]
As we shall see, this simplification still leads to accurate results for small values of $T$ and $\varepsilon$. At the same time, a systematic procedure for constructing the chart exists.

If we normalize parameter $T$ and eigenvalue parameter $\lambda$ as follows,
\[
\hat\lambda=\lambda\tau,\ \ \hat T=T/\tau,
\]
then the characteristic function takes the form
\[
H(\hat \lambda;\ \hat T,\varepsilon)\triangleq1-H_{11} e^{-\hat\lambda}+H_{11}\frac{1}{1+\hat\lambda \hat T} H_{11} e^{-\hat\lambda(1+\varepsilon)}.
\]
Assuming that (\ref{relgamma}) is satisfied,  we distinguish between two cases.

\paragraph{Case  $H_{11}\in(\frac{1}{2}\ ,1)$}\  \label{caspos}
We adopt the D-subdivision approach to find pairs $(\hat T,\varepsilon)$ for which there are imaginary axis characteristic roots.
We always have $H(0;\ \hat T,\varepsilon)=1$.  Looking for a characteristic root on the imaginary axis, $\lambda=\imath\omega$, with $\omega\in\mathbb{R}_{\geq 0}$, leads us to
\begin{equation}\label{twosides}
1-H_{11} e^{-\imath\omega}+H_{11} \frac{e^{-\imath\omega(1+\varepsilon)}}{1+\imath\omega \hat T}=0.
\end{equation}
Defining auxiliary variables  $(\theta,\Delta)\triangleq (\omega \hat T,\omega \varepsilon)$, we can rewrite the condition as
\begin{equation}\label{solutions}
1-\frac{1}{H_{11}} e^{\imath\omega}= \frac{e^{-\imath\Delta}}{1+\imath\theta},
\end{equation}
where $\omega$ has been separated from the other free parameters $(\Delta,\theta)$. Note that a solution is  possible if and only if the left-hand side has modulus smaller or equal to one.  

Let $\Omega_{0}=[\omega_1,\ \omega_2]\in \left(\frac{3\pi}{2},\ \frac{5\pi}{2}\right)$, with $\omega_2\neq \omega_1$, such that \begin{equation}\label{condmodulus}
\left|1-\frac{1}{H_{11}}e^{\imath\omega_i}\right|=1,\ \ i=1,2.
\end{equation}
 For every $\omega\in\Omega_0$, corresponding values for $\theta$ and $\Delta$ solving (\ref{solutions}) can be determined by matching modulus and argument of both sides, respectively. 
More specifically, let $\Omega_0\ni\omega\mapsto (\hat\theta(\omega),\hat\Delta(\omega) )$ be the continuous function describing all solutions for which $\Delta(\omega)\in (0,2\pi)$.  Such a continuous function always exist since
\[
\angle\left(\left( 1-\frac{1}{H_{11}}e^{\imath\omega}\right) \left(1+\imath\hat\theta(\omega)\right)  \right)\in(0,\pi]\cup(-\pi,0),\ \forall\omega\in\Omega_0.
\]
Let $\mathcal{S}_0= \left\{ \left( \hat\theta(\omega),\hat\Delta(\omega)\right):\ \omega\in\Omega_0  \right\}$ be the corresponding branch in the $(\theta,\Delta)$-parameter space. Other branches of solutions are then given by 
\[
\mathcal{S}_k\triangleq \left\{ \left( \hat\theta(\omega),\hat\Delta(\omega)+2\pi k\right):\ \omega\in\Omega_0  \right\},
\]
 with $k\in\mathbb{Z}$. Note that each element of $\mathcal{S}_{-1}$ has a strictly negative second component, which is apparent from Figure~\ref{figmother}, where we visualize the sets $\mathcal{S}_0$ and $\mathcal{S}_{-1}$ for $H_{11}=0.6$.  
 
 Condition $\left|1-\frac{1}{H_{11}}e^{\imath\omega}\right|\leq 1$ is also satisfied for the %half-open
 interval 
 $\omega\in (0,\ \omega_2-2\pi]$. This gives rise to branches
 \[
\mathcal{T}_k\triangleq \left\{ \left( \hat\theta(\omega),\hat\Delta(\omega)+2\pi k\right):\ \omega\in (2\pi,\omega_2]\right\}.
 \]

 \begin{figure}
 \begin{center}
 	\resizebox{8.5cm}{!}{\includegraphics{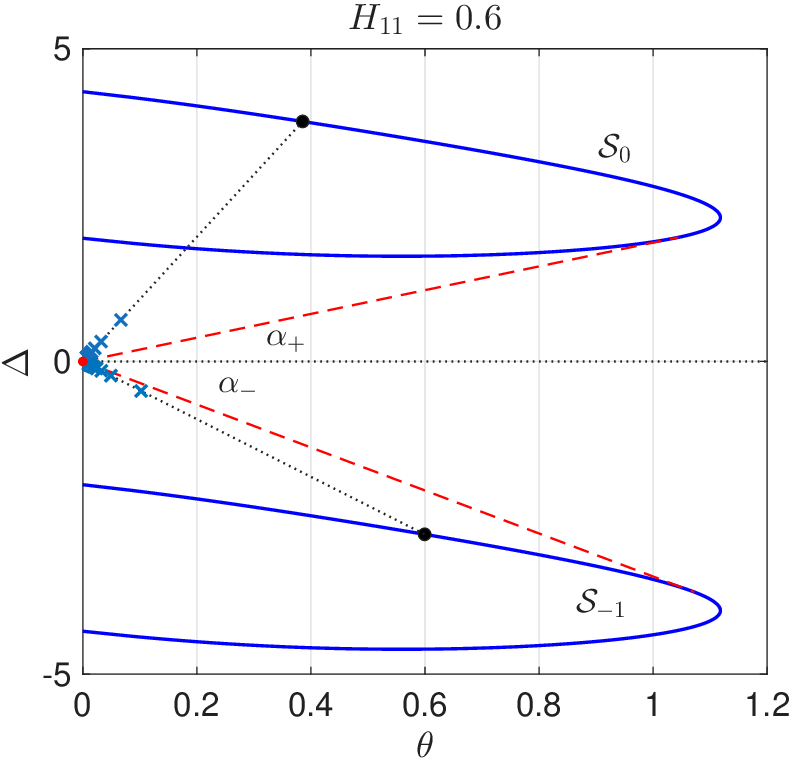}}
 	\caption{\label{figmother} Sets $\mathcal{S}_0$ and $\mathcal{S}_{-1}$ for $H_{11}=0.6$. The tick black dots represents $2$ points on $\cup_{k\in\mathbb{Z}} \mathcal{S}_k$, whose off-spring of critical points, described by (\ref{offspring}), are indicated with crosses.  The two dashed tangents to $\mathcal{S}_0$ and $\mathcal{S}_{-1}$, respectively, bound the largest cone that is contained in the stability region.}
 \end{center}
\end{figure}

 Now we are rea\,dy to construct the stability crossing curves in the $(\hat T,\varepsilon)$-parameter space.  Every point on $\bigcup_{k\in\mathbb{Z}} \mathcal{S}_k$ is characterized by one frequency $\omega\in [\omega_1,\ \omega_2]$ and coordinates $(\theta,\Delta)$, leading directly to a critical point $(\hat T,\varepsilon)=\left(\frac{\theta}{\omega},\frac{\Delta}{\omega} \right)$. Furthermore, the above derivation can be repeated for $\omega \in\Omega_\ell$, with $\Omega_\ell=[\omega_1+\ell 2\pi,\ \omega_2+\ell 2\pi]$ and $\ell\in\mathbb{Z}_{\geq 0}$, which leads to the same curves  $\mathcal{S}_k$.  Consequently, each point $(\theta,\Delta)\in \bigcup_{k\in\mathbb{Z}} \mathcal{S}_k$ and corresponding frequency $\omega\in[\omega_1,\ \omega_2]$ actually defines an infinite number of critical points, namely
 \begin{equation}\label{offspring}
 (\hat T,\varepsilon)=\left(\frac{\theta}{\omega+2 \pi\ell},\frac{\Delta}{\omega+ 2\pi\ell} \right),\ \ \ell \in\mathbb{Z}_{\geq 0}.
 \end{equation}
This is further illustrated in Figure~\ref{figmother}. Similarly, every point on $\bigcup_{k\in\mathbb{Z}} \mathcal{T}_k$ is characterized by one frequency $\omega\in (2\pi,\ \omega_2]$ and coordinates $(\theta,\Delta)$, which translate into a critical point
\[
(\hat T,\varepsilon)=\left(\frac{\theta}{\omega-2\pi},\frac{\Delta}{\omega-2\pi} \right).
\]
 Thus, the set of stability crossing curves in the $(\hat T,\varepsilon)$-parameter space consists of an infinite number of frequency-dependent scaled copies of the curves $\mathrm{S}_k$ and a scaled copy of the curves $\mathcal{T}_k$, creating a fractal-like structure around $(\hat T,\varepsilon)=(0,0)$. This is clarified with Figure~\ref{figcrossing}. Note that the stability crossing curves emanating from $\mathcal{S}_k$ are bounded, unlike the ones emanating from $\mathcal{T}_k$.

\begin{figure}
 \begin{center}
 	\resizebox{12cm}{!}{\includegraphics{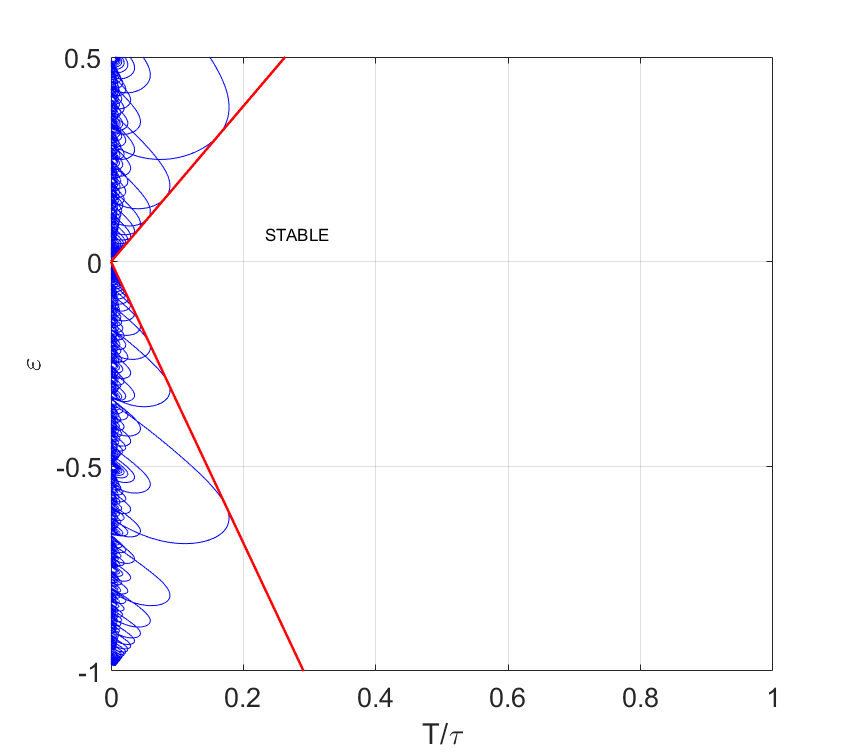 }}
 \resizebox{12cm}{!}{\includegraphics{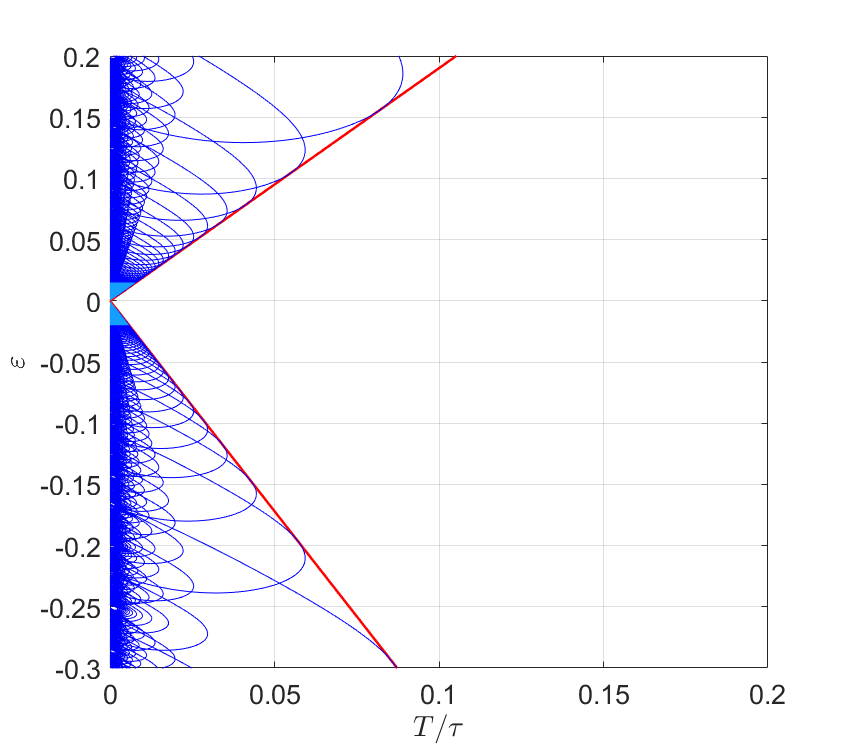 }}
 
\caption{\label{figcrossing} Stability crossing curves of (\ref{plantpde}) and (\ref{controlpde})  in the $(T,\varepsilon)$-parameter space  when neglecting the two distributed delay terms. The lower pane is obtained by zooming in on the upper pane. The scaling property (\ref{offspring}) induces infinitely many stability crossing curves in the neighborhood of $(T,\varepsilon)=(0,0)$, and a fractal nature of the plot. Only finitely many curves are shown, while the superimposed triangles indicate the regions with a high concentration.}
 \end{center}
\end{figure}

 %\begin{figure}
%	\begin{center}
%		\resizebox{10cm}{!}{\includegraphics{crossingcurves1.png}}\\
%			\resizebox{10cm}{!}{\includegraphics{crossingcurves2.png}}
%		\caption{\label{figcrossing} Stability crossing curves in the $(T,\varepsilon)$-parameter space for $H_{11}=0.6$. The lower pane is obtained by zooming in on the upper pane.}
%	\end{center}
%\end{figure}

 In Figures~\ref{figmother}-\ref{figcrossing}, we also indicate the largest cone that is contained in the stability region.  Finally, in Figure~\ref{fighoek} we highlight the dependence of the angles $\alpha_{\pm}$,  describing this cone, as a  function of $H_{11}$. Observe that
 \[
 \lim_{H_{11}\rightarrow \frac{1}{2}+} \alpha_+(H_{11})=+\infty,\  \lim_{H_{11}\rightarrow \frac{1}{2}+} \alpha_-(a)=-\infty,\ 
 \] 
 which is consistent with the fact that for $\left|H_{11}\right|<\frac{1}{2}$ and $T=0$ the system is delay-independent stable (i.e., stable for all $\tau$ and $\varepsilon$). At the same time, we have
 \[
 \lim_{H_{11}\rightarrow 1-} (\alpha_+(H_{11})- \alpha_-(H_{11}) )= 0,
 \]
 which is also natural as for $\left|H_{11}\right|>1$,  the system is always unstable due to instability of the associated delay-difference equation $x(t)=H_{11} x(t-\tau)$.

   \begin{figure}
   	\begin{center}
   		\resizebox{11cm}{!}{\includegraphics{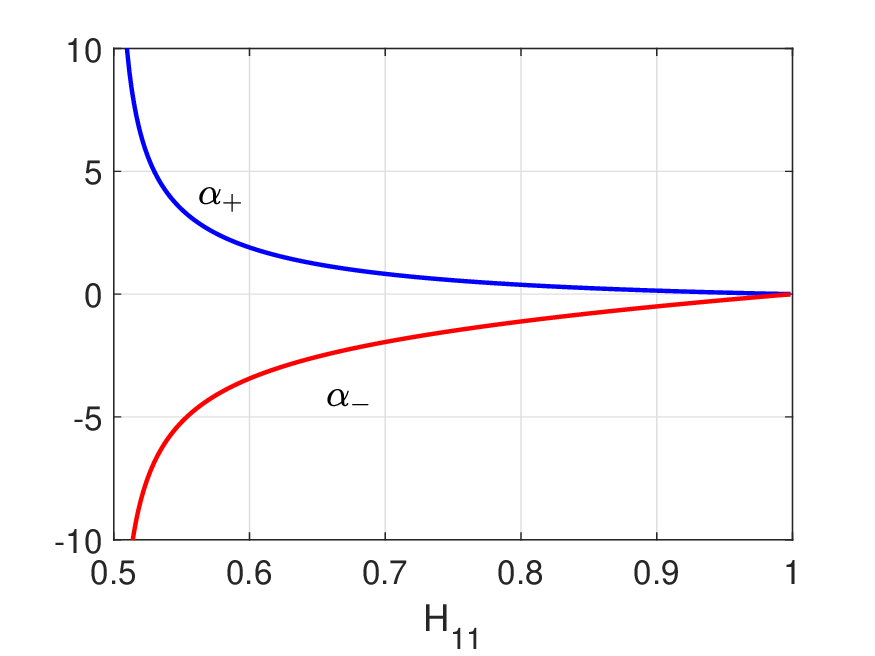}}
   		\caption{\label{fighoek} Angles $\alpha_+$ and $\alpha_-$ as a function of parameter $H_{11}$.}
   	\end{center}
   \end{figure}

 \paragraph{Case $H_{11}\in(-1,\ -\frac{1}{2})$}\
 The analysis is similar to the one in the previous case.  Neither the curves $\bigcup_{k\in\mathbb{Z}} \mathcal{S}_k$ nor the angles $\alpha_+,\ \alpha_-$ depend on the sign of $H_{11}$. The main difference is the parametrization in $\omega$. Now condition (\ref{condmodulus}) is not satisfied for an interval starting with $\omega=0$. As a consequence, no distinction as between curves $\mathcal{S}_k$ and $\mathcal{T}_k$ needs to be made, while all stability crossing curves are bounded.
 
 %In Figure~\ref{figcrossingb} the stability crossing curves are visualized for $a=-\frac{3}{4}$.
%\begin{figure}
% 	\begin{center}
% 		\resizebox{10cm}{!}{\includegraphics{crossingcurves1b.png}}
% 		\caption{\label{figcrossingb} Stability crossing curves in the $(T,\varepsilon)$-parameter space for $a=-\frac{3}{4}$.}
% 	\end{center}
% \end{figure}

\subsubsection{Taking into account distributed delays}
%parameters from the submitted SICON paper.
%We consider the model mismatch
%\[
%\lambda_1/(1+\varepsilon), \mu_1/(1+\varepsilon), \sigma_{12}/(1+\varepsilon),\ \sigma_{21}/(1+\varepsilon).
%\]

For numerical values (\ref{numvalpde}), we have computed the stability region of (\ref{plantpde}) and (\ref{controlpde}) in the $(T,\varepsilon)$-parameter space using numerical continuation  of branches in the $(T,\varepsilon,\omega)$-space that characterize the presence of imaginary axis roots $\pm\imath\omega$, followed by projection on the $(T,\varepsilon)$-plane, since a semi-analytic approach is no longer viable.  As the stability crossing curves need to be continued one-by-one, we have restricted ourselves to these curves adjacent to the central stability region.

Figure~\ref{figstabdist} shows the obtained stability region in the $(T,\varepsilon)$-parameter space. For comparison, in Figure~\ref{figcrossing} the stability region is visualized for $w_1=w_{\mathrm{des}}\equiv 0$. The correspondence for small $(T,\varepsilon)$ can be explained as follows. Since the closed-loop system is exponentially stable for $T=\varepsilon=0$, with and without taking the distributed delays into account, destabilizing perturbations for small values of $(T,\varepsilon)$ must always be characterized by unstable characteristic roots with very large imaginary parts. This is also apparent from the derivation in \S\ref{parsmallT} and particularly expression (\ref{offspring}), which shows that the smaller the critical values of $(T,\varepsilon)$ are, the higher is  their frequency (imaginary part). Distributed delay terms are then characterized by highly oscillatory integrands that are almost integrated out.

Note further that the stability region in Figure~\ref{figstabdist} is stretched in the direction of increasing $T$ and decreasing $\varepsilon$, since the underestimation of the plant delay for negative $\varepsilon$ is partly compensated by the lag induced by the filter.

\begin{figure}
 \begin{center}
 	\resizebox{12cm}{!}{\includegraphics{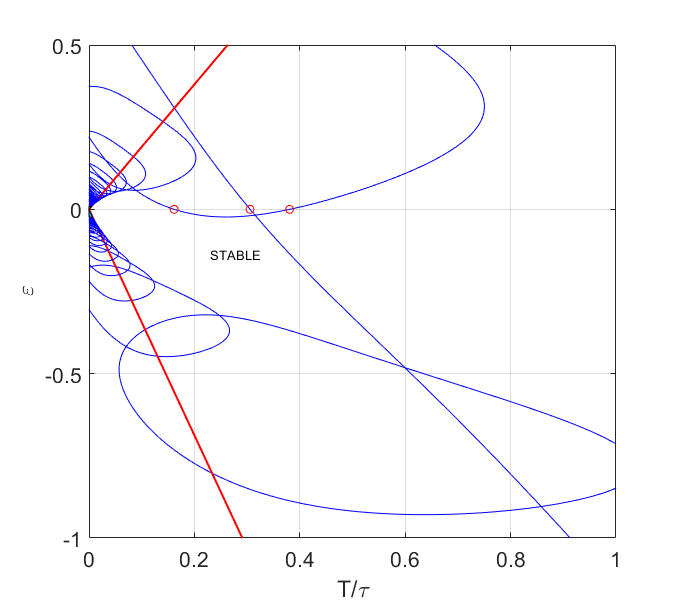 }}
 	\resizebox{12cm}{!}{\includegraphics{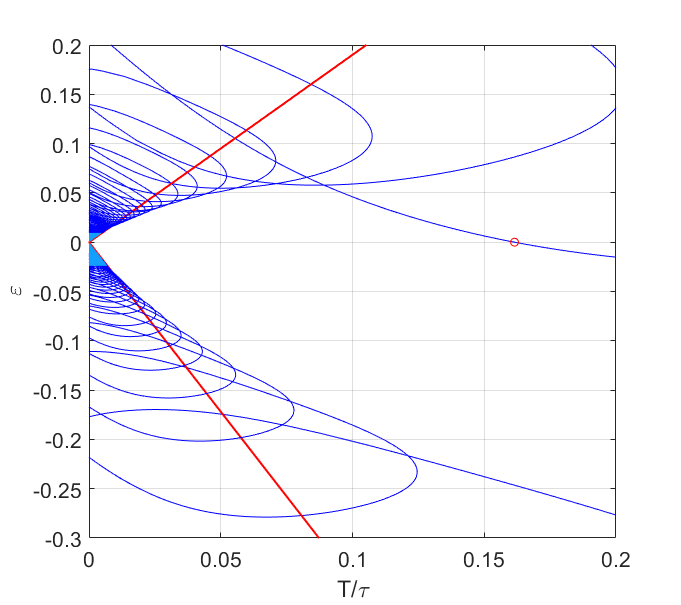 }}
 	\caption{Stability region of (\ref{plantpde}) and (\ref{controlpde}) in the $(T,\varepsilon)$-parameter space. Only the stability crossing curves in the region adjacent to central stability region are shown. For $\varepsilon=0$, the critical values of $T$, computed by Algorithm~\ref{algsweep}, are indicated by circle markers. The lower pane is obtained by zooming in on the upper pane. The superimposed triangles indicate the regions with a high concentration of stability crossing curves.
 	\label{figstabdist} }
 \end{center}
\end{figure}

\section{Conclusions}\label{secconcl}
Using a combined approach, relying on delay-based modeling of PDEs, robust control techniques,  as well as the analytic and numerical construction of stability charts, we made a qualitative and quantitative analysis of the choice of the cut-off frequency of the filter, used to mitigate the  fragility of control law (\ref{control-nominal}), in appropriately balancing robustness at high frequencies with nominal performance at low frequencies. We assumed that all the velocities of the PDE system~\eqref{eq:hyperbolic_couple} are distinct. However, the case of identical velocities can be treated in a similar manner. More precisely, in the presence of ``isotachic'' blocks (i.e., groups of states sharing the same transport speed), it was shown in~\cite{hu2015control} that a suitable diagonalizing transformation can be performed to decouple the isotachic states. After applying this preliminary change of variables, the backstepping methodology can still be used to rewrite the system in the form~\eqref{sys-final}.

Future work includes the adaptation of numerical methods for control of systems governed by delay-differential algebraic equation (DDAE) models towards hyperbolic PDE systems, with particular attention to the optimal shaping of distributed delay kernels in the controls, and the integrated design of filters and controllers.

\section*{Acknowledgments}
This work received funding from the Internal Funds of KU Leuven, project C14/22/092, and from the French government under the France 2030 program, reference ANR-11-IDEX-0003 within the OI H-Code.

\appendix

\subsection{Remodeling as an IDE }\label{apA}

%\subsubsection{Nominal stabilizing backstepping controller}
In this section, we recall the design of the backstepping controller introduced in~\cite{Auriol2019a,coron2017finite} to stabilize the system~\eqref{eq:hyperbolic_couple}. In particular, we emphasize the connections between~\eqref{eq:hyperbolic_couple} and the IDE system~\eqref{sys-final}.

\subsubsection{First backstepping transformation: removing the in-domain couplings}
Consider the backstepping transformation introduced in~\cite{coron2017finite,hu2019boundary}
\begin{equation}
    \alpha(t,y)=z(t,y)-\int_0^y K(y,\xi) z(t,\xi)\,d\xi, \label{BS_transf}
\end{equation}
where $\alpha(t,y)=(\alpha^\top_+(t,y),\alpha^\top_-(t,y))^\top \in\mathbb{R}^{k+\ell}$ and
\[
K(y,\xi)=\begin{pmatrix}
    K^{++}(y,\xi)& K^{+-}(y,\xi)\\
    K^{-+}(y,\xi)& K^{--}(y,\xi)
\end{pmatrix},
\]
with $K^{++} \in \mathbb{R}^{k\times k}$, $K^{+-} \in \mathbb{R}^{k\times \ell}$, $K^{-+} \in \mathbb{R}^{\ell\times k}$, and $K^{--} \in \mathbb{R}^{\ell\times \ell}$ defined on the triangular domain $\mathcal{T}=\{(y,\xi)\in [0,1]^2:~y\geq \xi\}$.  
The kernels satisfy the PDE system
\begin{align}
    &\Lambda  K_y(y,\xi)+ K_\xi(y,\xi)\Lambda=- K(y,\xi)\Sigma,\label{eq:K_PDE} \\
    &K(y,y)\Lambda-\Lambda K(y,y)=-\Sigma,\\
    &(K^{--}(y,0)\Lambda^-)_{ij}=(K^{-+}(y,0)\Lambda^+Q)_{ij}, \quad i\geq j, \label{eq:K_--}
\end{align}
with boundary conditions $K_{ij}^{++}(y,0)=0$ for $i\geq j$, $K_{ij}^{++}(1,\xi)=0$ for $i<j$, and $K_{ij}^{--}(1,\xi)=-\tfrac{\Sigma^{--}_{ij}}{\mu_i-\mu_j}$ for $i<j$. It was shown in~\cite{di2018stabilization,hu2019boundary} that this kernel system admits a unique bounded piecewise-continuous solution. The transformation~\eqref{BS_transf} is thus a Volterra transformation, and therefore boundedly invertible~\cite{yoshida1960lectures}. Its inverse reads
\begin{equation}
    z(t,y)=\alpha(t,y)+\int_0^yL(y,\xi)\alpha(t,\xi)\,d\xi, \label{eq:inv_backstepping}
\end{equation}
where $L$ satisfies the Volterra equation
\begin{equation}
   L(y,\xi)=K(y,\xi)+\int_\xi^yK(y,\eta)L(\eta,\xi)\,d\eta,\quad (y,\xi)\in\mathcal{T}. \label{eq:def_L}
\end{equation}

The transformation~\eqref{BS_transf} maps the system~\eqref{eq:hyperbolic_couple} into the target system defined for all $t>0$ and all $y \in[0,1]$ by
\begin{equation} \label{eq:hyperbolic_target_1}
\left\{
\begin{aligned}
&\alpha_t(t, y)+\Lambda \alpha_y(t, y)=\begin{pmatrix}
    0&G^{-+}(y)\\ 0&G^{--}(y)
\end{pmatrix} \alpha(t,0),  \\
&\alpha_+(t, 0)=Q \alpha_-(t, 0),\\
&\alpha_-(t, 1)=R \alpha_+(t, 1)+u(t)+\int_0^1N(\xi)\alpha(t,\xi)\,d\xi,
\end{aligned}
\right.
\end{equation}
with
\begin{align}
    G^{-+}(y)&=K^{-+}(y,0)\Lambda^--K^{++}(y,0)\Lambda^+Q, \label{eq:def_G_-+}\\
    G^{--}(y)&=K^{--}(y,0)\Lambda^--K^{+-}(y,0)\Lambda^+Q.\label{eq:def_G_--}
\end{align}
Due to~\eqref{eq:K_--}, the matrix $G^{--}$ is strictly upper-triangular. The function $N$ is defined as 
\[
N(\xi)=\begin{pmatrix}
   N^+(\xi) & N^-(\xi) 
\end{pmatrix}, \qquad \xi\in[0,1],
\]
with
\begin{align}
    N^+(\xi)&=-L^{-+}(1,\xi)+RL^{++}(1,\xi), \label{eq:def_N_+}\\
    N^-(\xi)&=-L^{--}(1,\xi)+RL^{-+}(1,\xi). \label{eq:def_N_-}
\end{align}

Through this (invertible) Volterra transformation, part of the in-domain couplings has been shifted to the boundary (via the integral terms $N^+$ and $N^-$), while the remaining couplings appear as nonlocal terms. These terms can be slightly modified to rewrite the system  a difference system.

\subsubsection{Second integral transformation: removal of nonlocal terms}
We now apply the integral transformation from~\cite{coron2017finite}
\begin{equation}
    \alpha(t,y)=w(t,y)-\int_0^1\begin{pmatrix}
        0_{k\times k} & 0_{k\times \ell}\\
        0_{\ell\times k} & F(y,\xi)
    \end{pmatrix}w(t,\xi)\,d\xi, \label{eq:backstepping_last}
\end{equation}
where $w(t,y)=(w^\top_+(t,y),w^\top_-(t,y))^\top \in\mathbb{R}^{k+\ell}$ and $F \in \mathbb{R}^{\ell\times \ell}$ is strictly upper-triangular on $[0,1]^2$, solving
\begin{align}
&\Lambda^- F_y(y,\xi)+F_\xi(y,\xi)\Lambda^-=0,\label{eq:F_PDE}\\
&F(y,0)=G^{--}(y)(\Lambda^-)^{-1},\quad F(0,\xi)=0,\label{eq:F_BC}
\end{align}
with $\Lambda^-=\mathrm{diag}(\mu_1,\dots,\mu_m)$.  
This PDE system admits a unique piecewise-continuous solution, obtained explicitly by the method of characteristics~\cite{coron2017finite}. Owing to its cascade structure, transformation~\eqref{eq:backstepping_last} is affine and invertible. It maps~\eqref{eq:hyperbolic_target_1} into the target system defined for all $t>0,~ y \in[0,1]$ by
\begin{equation} \label{eq:hyperbolic_target_2}
\left\{
\begin{aligned}
&w_t(t, y)+\Lambda w_y(t, y)=\begin{pmatrix}
    G^{-+}(y)w_-(0)\\ \bar G^{--}(y)w_-(1)
\end{pmatrix}, \\
&w_+(t, 0)=Q w_-(t, 0),\\
&w_-(t, 1)=R w_+(t, 1)+u(t)+\int_0^1N_w(\xi)w(t,\xi)\,d\xi, 
\end{aligned}
\right.
\end{equation}
where $\bar G^{--}$ solves the Volterra equation
\begin{equation}
   \bar G^{--}(y)=F(y,1)\Lambda^-+\int_0^1F(y,\xi)\bar G^{--}(\xi)\,d\xi, \label{eq:def_G_bar}
\end{equation}
and
\[
    N_w(\xi)=\begin{pmatrix}
        N^+(\xi)& \hspace{-0.25cm}F(1,\xi)+N^-(\xi)-\int_0^1N^-(\nu)F(\nu,\xi)\,d\nu
    \end{pmatrix}.
\]

\subsubsection{Time-delay representation and control law achieving fixed-time stability} \label{parfixedtime}
Using the method of characteristics,~\cite{Auriol2019a} shows that $w_-(t,1)$ satisfies an integral delay equation of the form~\eqref{sys-final}, with matrices $H_i$ and $w_1$ expressed explicitly in terms of the PDE parameters, and $\tau=\tfrac{1}{\lambda_1}+\tfrac{1}{\mu_1}$. 
%The IDE system~\eqref{sys-final} serves as a comparison model sharing equivalent stability properties with the PDE~\eqref{eq:hyperbolic_target_2}.
%\begin{lemma}(\cite[Theorem 6.3.1]{auriol2024contributions})\label{Lemma_equiv_stab}
%    Exponential stability of~\eqref{sys-final} in the sense of Definition~\ref{def_exp_stab_IDE} is equivalent to exponential stability of~\eqref{eq:hyperbolic_couple} in the sense of Definition~\ref{Def_stability}.
%\end{lemma}
%\noindent\textbf{Proof.} See~\cite[Theorem 6.3.1]{auriol2024contributions}. The result follows from the bounded invertibility of the successive integral transformations. \hfill $\Box$

From~\eqref{eq:hyperbolic_target_2}, a fixed-time stabilizing controller was designed in~\cite{coron2017finite}:
\begin{align}
    u(t)=-Rw_+(t,1)-\int_0^1 N_w(\xi)w(t,\xi)\,d\xi, \label{eq:control_nominal}
\end{align}
which enforces $w_-(t,1)=0$, and thus guarantees finite-time stability of both~\eqref{eq:hyperbolic_target_2} and~\eqref{eq:hyperbolic_couple}. In IDE form, this law becomes (\ref{control-nominal}), with $w_2=w_1$.
%\begin{align}
%    u(t)=-\sum_{i=1}^k \sum_{j=1}^{\ell} H_{ij} x\!\left(t-%(r_i+s_j)\right)+\int_{-\tau}^0 w_1(\theta)\, x(t+%\theta)\,d\theta,
%\end{align}
%again showing finite-time stability of the closed loop system.

%In \cite{\eqref{eq:hyperbolic_target_2}} However, this nominal design neglects robustness. As highlighted in~\cite{Auriol2019a,auriol2020robust}, canceling all reflection terms may yield zero delay-robustness margins and poor tolerance to parameter uncertainties, making~\eqref{eq:control_nominal} impractical. This limitation stems from the non-strictly-proper nature of the control operator. Alternative strategies have been proposed: partial reflection cancellation~\cite{auriol2018delay}, or the addition of a well-tuned low-pass filter~\cite{Auriol2023} to enforce strict properness and thus delay-robustness. %The latter leverages Assumption~\ref{assum_robustness}, which ensures natural open-loop exponential stability at high frequencies. 
%Using such a filter provides extra flexibility: high performance at low frequencies while guaranteeing robustness margins at high frequencies, thereby enabling an effective trade-off between performance and robustness. {\color{blue}This last paragraph needs to be rewritten depending on what is already written in the introduction.}

\subsection{Proof of Lemma~\ref{lemma:convergence_w}}
\label{secapB}
The overall proof consists of three main steps.

\subsubsection{Representation of kernel equations~\eqref{eq:K_PDE}--\eqref{eq:K_--} as integral equations}

%In this section, we show how the kernel %equations~\eqref{eq:K_PDE}--\eqref{eq:K_--} can be %rewritten as integral equations. This will be %essential for the proof of %Lemma~\ref{lemma:convergence_K}. 
Following the computations introduced in~\cite{hu2015control}, we apply the method of characteristics. For each $1\leq i \leq \ell$ and $1\leq j \leq k$, and all $(y,\xi)\in \mathcal{T}$, we have
\begin{align}
 &K^{-+}_{ij}(y,\xi)=-\frac{\Sigma_{ij}^{-+}}{\lambda_j+\mu_i} +\int_0^{\frac{y-\xi}{\lambda_j+\mu_i}} \Big[\sum_{p=1}^k\Sigma_{pj}^{++}K^{-+}_{ip}(y-\mu_is,\nonumber \\
 &\xi+\lambda_js)+\sum_{p=1}^\ell\Sigma_{pj}^{-+}K^{--}_{ip}(y-\mu_is,\xi+\lambda_js)\Big] \,ds.  
\end{align}
Also, for each $1\leq i \leq \ell$ and $1\leq j \leq \ell$, and all $(y,\xi)\in \mathcal{T}$
\begin{multline}
 K^{--}_{ij}(y,\xi)=-\frac{\Sigma_{ij}^{--}}{\mu_i-\mu_j} \gamma_{ij}(y,\xi)-(1-\gamma_{ij}(y,\xi))\frac{1}{\mu_j}\sum_{p=1}^k \lambda_p \\ q_{pj}\frac{\Sigma_{ip}^{-+}}{\lambda_p+\mu_i}
 +(1-\gamma_{ij}(y,\xi))\frac{1}{\mu_j}\sum_{p=1}^k \lambda_p q_{pj}\int_0^{\frac{\mu_i}{\mu_j}\frac{\xi}{\lambda_p+\mu_i}} \Big[\sum_{r=1}^k\Sigma_{rp}^{++}\\
 K^{-+}_{ir}(y-\mu_is,\xi+\lambda_ps)
 +\sum_{r=1}^\ell\Sigma_{rp}^{-+}K^{--}_{ir}(y-\mu_is,\xi+\lambda_ps)\Big] \,ds\\-\upsilon_{ij}\int_0^{s^F_{ij}(y,\xi)} \Big[\sum_{p=1}^\ell\Sigma_{pj}^{+-}K^{-+}_{ip}(y+\upsilon_{ij}\mu_is,\xi+\upsilon_{ij}\mu_js)\\+\sum_{p=1}^\ell\Sigma_{pj}^{--}K^{--}_{ip}(y+\upsilon_{ij}\mu_is,\xi+\upsilon_{ij}\mu_js)\Big] \,ds,  
\end{multline}
where 
\begin{equation}
   s^F_{ij}= \left\{
\begin{aligned}
\frac{\xi}{\mu_j}, \quad &\text{if $i\geq j$ and $\mu_i \xi-\mu_j y\leq 0$,}\\
\frac{y-\xi}{\mu_i-\mu_j}, \quad &\text{if $i > j$ and $\mu_i \xi-\mu_j y> 0$,}\\
\frac{y-\xi}{\mu_j-\mu_i},\quad &\text{if $i< j$ and $ \xi\geq  1-\frac{\mu_j}{\mu_i}(1-y)$,}\\
\frac{1-y}{\mu_i},\quad &\text{if $i< j$ and $ \xi<  1-\frac{\mu_j}{\mu_i}(1-y)$,}
\end{aligned}
\right.
\end{equation}

\begin{equation}
   \upsilon_{ij}= \left\{
\begin{aligned}
1, \quad \text{if $i<j$,}\\
-1\quad \text{if $i\geq j$,}
\end{aligned}
\right.
\end{equation}
and 
\begin{equation}
   \gamma_{ij}(y,\xi)= \left\{
\begin{aligned}
0, \quad &\text{if $i\geq j$ and $\mu_i \xi-\mu_j y\leq 0$,}\\
1,\quad &\text{otherwise.}
\end{aligned}
\right.
\end{equation}
Let us denote $\mathbf{K}$ the vector containing all the kernels $K_{ij}^{-+}$ and $K^{--}_{ij}$, reordered line by line and stacked up. We have
\begin{equation}
    \mathbf{K}=\begin{pmatrix}
        K^{-+}_{11}, \cdots, K_{mn}^{-+}, K^{--}_{11}, \cdots K_{mm}^{--}
    \end{pmatrix}^\top.
\end{equation}
The vector $\mathbf{K}$ is the solution of the following integral equation
\begin{equation}
    \label{eq:integ_eq_H}
    \mathbf{K}=\varphi+\Phi(\mathbf{K}),
\end{equation}
with 
\begin{equation}
    \varphi=\begin{pmatrix}
        \varphi_{1,1},\cdots \varphi_{m,n},\varphi_{1+n,n+1},\cdots \varphi_{m+n,k+\ell}
    \end{pmatrix}^\top
\end{equation}
\begin{equation}
    \Phi(\mathbf{K})=\begin{pmatrix}
        \Phi_{11}(\mathbf{K}),\cdots \Phi_{mn}(\mathbf{K}),\Phi_{1+n,n+1}(\mathbf{K}),\cdots
    \end{pmatrix}^\top % \Phi_{m+n,k+\ell}(\mathbf{K})
\end{equation}
where for all $1\leq i\leq \ell$ and all $1\leq j\leq k$
\begin{align}
    \varphi_{ij}&=-\frac{\Sigma_{ij}^{-+}}{\lambda_j+\mu_i},\label{eq:phi_case1} \\
    \Phi_{ij}(\mathbf{K})(y,\xi)&=\int_0^{\frac{y-\xi}{\lambda_j+\mu_i}} \Big[\sum_{p=1}^k\Sigma_{pj}^{++}K^{-+}_{ip}(y-\mu_is,\xi+\lambda_js)\nonumber \\
 &+\sum_{p=1}^\ell\Sigma_{pj}^{-+}K^{--}_{ip}(y-\mu_is,\xi+\lambda_js)\Big] \,ds,
\end{align}
and where for all $1\leq i\leq \ell$ and all $1\leq j\leq \ell$
\begin{multline}
    \varphi_{\ell+i,k+j}=-\frac{\Sigma_{ij}^{--}}{\mu_i-\mu_j} \gamma_{ij}(y,\xi)\\-(1-\gamma_{ij}(y,\xi))\frac{1}{\mu_j}\sum_{r=1}^k \lambda_r q_{rj}\frac{\Sigma_{ir}^{-+}}{\lambda_r+\mu_i},\\
    \Phi_{\ell+i,k+j}(\mathbf{K})(y,\xi)=(1-\gamma_{ij}(y,\xi))\frac{1}{\mu_j}\sum_{r=1}^k \lambda_r q_{rj}\int_0^{\frac{\mu_i}{\mu_j}\frac{\xi}{\lambda_r+\mu_i}} \\\Big[\sum_{p=1}^k\Sigma_{pr}^{++}K^{-+}_{ip}(y-\mu_is,\xi+\lambda_rs)+\sum_{p=1}^\ell\Sigma_{pr}^{-+}K^{--}_{ip}(y-\mu_is,\\
    \xi+\lambda_rs)\Big] \,ds-\upsilon_{ij}\int_0^{s^F_{ij}(y,\xi)} \Big[\sum_{r=1}^\ell\Sigma_{rj}^{+-}
 K^{-+}_{ir}(y+\upsilon_{ij}\mu_is,\xi+\\\upsilon_{ij}\mu_js)+\sum_{r=1}^\ell\Sigma_{rj}^{--}K^{--}_{ir}(y+\upsilon_{ij}\mu_is,\xi+\upsilon_{ij}\mu_js)\Big] \,ds. \label{eq:Phi_case2} 
\end{multline}

Similar integral equations can be obtained for the kernels $K^{++}$ and $K^{+-}$.

\subsubsection{Impact of model mismatch  on backstepping kernels}
Let $\delta \in(0,\delta^\star)$, with $\delta^\star$ determined by Proposition~\ref{prop:closed_ball}. For any $\delta \in [0,\hat \delta]$ and $\mathcal{\bar P}\in\mathcal{B}_{\mathcal{P}}^\delta$, we denote $\bar{K}$ the backstepping kernel computed from~\eqref{eq:K_PDE}--\eqref{eq:K_--} using the parameters $\bar{\Lambda}$, $\bar{\Sigma}$, $\bar{Q}$, and $\bar{R}$.
We state and subsequently prove the following lemma.
\begin{lemma} \label{lemma:convergence_K}
  There exists a constant $C_K>0$ such that for any $\delta\in[0,\hat \delta]$  and $\mathcal{\bar P}\in\mathcal{B}_{\mathcal{P}}^\delta$, for all $(y,\xi) \in \mathcal{T}$, for all $1\leq i,j\leq k+\ell$
\begin{equation}\label{eq:ineg_K}
|\bar{K}_{ij}(y,\xi)-K_{ij}(y,\xi)| \leq  C_K \big(\delta+\mathds{1}_{\bar{D}_{ij}}(y,\xi)\big),
\end{equation}
where $\bar{D}_{ij}$ is a subset of $\mathcal{T}$ satisfying $\nu(\bar{D}_{ij})\leq C_K\delta$.  
\end{lemma}

Let $\delta\in[0,\hat \delta]$  and $\mathcal{\bar P}\in\mathcal{B}_{\mathcal{P}}^\delta$.
In what follows, we denote by $C_i$ positive constants that do not depend on $\delta$. 
To prove Lemma~\ref{lemma:convergence_K}, we analyze the effect of model mismatch on the kernel $K$. We will first focus on the component $K^{-+}$ and $K^{--}$. With model mismatch, the new kernels $\bar{K}^{-+}$ and $\bar{K}^{--}$ satisfy the integral equation
\begin{equation}
    \label{eq:integ_eq_H_bis} \mathbf{\bar{K}}=\bar{\varphi}+\bar{\Phi}(\mathbf{\bar{K}}),
\end{equation}
where $\mathbf{\bar{K}}$ is the vector containing all the kernels $\bar{K}^{-+}_{ij}$ and $\bar{K}^{--}_{ij}$, reordered line by line and stacked up
\begin{equation}
    \mathbf{\bar{K}}=\begin{pmatrix}
        \bar{K}^{-+}_{11}, \cdots, \bar{K}^{-+}_{mn}, \bar{K}^{--}_{11}, \cdots \bar{K}^{--}_{mm}
    \end{pmatrix}^\top,
\end{equation}
and where $\bar{\varphi}$ and $\bar{\Phi}$ satisfy analogous equations to~\eqref{eq:phi_case1}-\eqref{eq:Phi_case2}, the parameters $\Lambda, \Sigma, Q$ and $R$ being replaced by  $\bar{\Lambda}, \bar{\Sigma}, \bar{Q}$ and $\bar{R}$. Let us denote the difference between the two kernels as $\tilde{\mathbf{K}}=\mathbf{\bar{K}}-\mathbf{K}$. Direct computations yield
\begin{equation}\label{eq:K_tilde}
  \tilde{\mathbf{K}}=\tilde \varphi+\bar{\Phi}(\tilde{\mathbf{K}}),  
\end{equation}
where 
\begin{equation} \label{eq:phi_tilde}
    \tilde \varphi=\underset{\psi}{\underbrace{\bar{\varphi}-\varphi}}+\underset{\tilde{\Phi}(\mathbf{K})}{\underbrace{\bar{\Phi}(\mathbf{K})-\Phi(\mathbf{K})}}.
\end{equation}
In what follows, we will analyze the effect of parameter mismatch of $\psi$ and $\tilde \Phi(\mathbf{K})$. We will then use the method of successive approximation to prove Lemma~\ref{lemma:convergence_K}.

\paragraph*{Analysis of $\psi$} We can rewrite the vector $\psi$ as follows
\begin{equation}
    \psi=\begin{pmatrix}
        \psi_{11},\cdots \psi_{mn},\psi_{1+n,n+1},\cdots \psi_{m+n,k+\ell}
    \end{pmatrix}^\top.
\end{equation}
For all $1\leq i\leq \ell$, all $1\leq j\leq k$, $\psi_{ij}$, and all $(x,y)\in\mathcal{T}$, we have
\begin{align*}
    |\psi_{ij}(y,\xi)|&=\left|\frac{\bar{\Sigma}_{ij}^{-+}}{\bar{\lambda}_j+\bar{\mu}_i}-\frac{\Sigma_{ij}^{-+}}{\lambda_j+\mu_i}\right|\nonumber \\
    &=\left|\frac{\bar{\Sigma}^{-+}_{ij}-\Sigma_{ij}^{-+}}{\lambda_j+\mu_i}+\bar{\Sigma}_{ij}^{-+}\frac{(\lambda_j-\bar{\lambda}_j)+(\mu_i-\bar{\mu}_i)}{(\lambda_j+\mu_i)(\bar{\lambda}_j+\bar{\mu}_i)}\right|.
\end{align*}
Using {Definition~\ref{def:ball}}, we obtain 
\begin{align}
     |\psi_{ij}(x,y)|&\leq C_1 \delta.
\end{align}

For all $1\leq i\leq \ell$, all $1\leq j\leq \ell$, and all $(y,\xi)\in\mathcal{T}$, we obtain, using analogous estimates
\begin{multline}
|\psi_{m+i,n+j}(y,\xi)|\leq 
\left|(\bar{\gamma}_{ij}(y,\xi)-\gamma_{ij}(y,\xi))\frac{\Sigma_{ij}^{--}}{\mu_i-\mu_j}\right|\nonumber\\
+\,\left|\bar{\gamma}_{ij}(y,\xi)\Big(\frac{\bar{\Sigma}_{ij}^{--}}{\bar{\mu}_i-\bar{\mu}_j}-\frac{\Sigma_{ij}^{--}}{\mu_i-\mu_j}\Big)\right|+\,\Bigg|(1-\bar{\gamma}_{ij}(y,\xi))\nonumber \\
\Big[\frac{1}{\mu_j}\sum_{r=1}^k\lambda_r q_{rj}\frac{\Sigma_{ir}^{-+}}{\lambda_k+\mu_i}
-\frac{1}{\mu_j}\sum_{r=1}^k\bar{\lambda}_r \bar{q}_{rj}\frac{\bar{\Sigma}_{ir}^{-+}}{\bar{\lambda}_r+\bar{\mu}_i}\Big]\Bigg|\nonumber \\
+\,\left|(\bar{\gamma}_{ij}(y,\xi)-\gamma_{ij}(y,\xi))\frac{1}{\mu_j}\sum_{r=1}^k\lambda_r q_{rj}\frac{\Sigma_{ir}^{-+}}{\lambda_r+\mu_i}\right|.   
\end{multline} 

We get
\begin{equation}\label{eq:psi_bound}
    |\psi_{m+i,n+j}(y,\xi)|\leq C_2\big(\bar{\gamma}_{ij}(y,\xi)-\gamma_{ij}(y,\xi)\big)+C_3\delta,
\end{equation}
where, for all $(y,\xi)\in \mathcal{T}$, we have 
\begin{equation}
   \bar{\gamma}_{ij}(y,\xi)-\gamma_{ij}(y,\xi)= \left\{
\begin{aligned}
1,\quad &\text{if $i> j$, $\mu_i\xi-\mu_jy\leq 0$},\\
&\text{ and $\bar{\mu}_i\xi-\bar{\mu}_jy> 0$,} \\
-1,\quad &\text{if $i> j$, $\mu_i\xi-\mu_jy> 0$},\\
&\text{and $\bar{\mu}_i\xi-\bar{\mu}_jy\leq 0$,} \\
0, \quad & \text{otherwise.}
\end{aligned}
\right.
\end{equation}
The function $\bar{\gamma}_{ij}(y,\xi)-\gamma_{ij}(y,\xi)$ does not necessarily go to zero when $\delta \rightarrow 0$. However, the measure of the domain where it is nonzero goes to zero as $\delta \rightarrow 0$, as illustrated in Figure~\ref{Fig_D_Delta}. More precisely, we obtain for all $(y,\xi)\in\mathcal{T}$,
\begin{equation}\label{eq:gamma_bound}
    |\bar{\gamma}_{ij}(y,\xi)-\gamma_{ij}(y,\xi)|\leq \mathds{1}_{\bar{D}_{ij}}(y,\xi),
\end{equation}
where $\nu(\bar{D}_{ij})\leq C_4 \delta$. All in all, for all $(y,\xi)\in \mathcal{T}$, $|\psi_{m+i,n+j}(y,\xi)|$ can be bounded by the sum of a constant multiplied by $\delta$ and a bounded function supported on a set whose Lebesgue measure decreases linearly with~$\delta$.

\begin{figure*}[htb]
\begin{center}
\scalebox{1}{
\begin{tikzpicture}[scale=1.7]

  % Axes
  \draw[->] (-0.2,0) -- (3.3,0) node[right] {$y$};
  \draw[->] (0,-0.2) -- (0,3.3) node[above] {$\xi$};

  % Sommets
  \coordinate (A) at (0,0);
  \coordinate (B) at (3,0);
  \coordinate (C) at (3,3);

  % Triangle
  \draw[thick,black] (A) -- (B) -- (C) -- cycle;
  \node at (2.2,1.8) {$\mathcal{T}$};

  % Graduations
  \node[below left] at (A) {0};
  \node[below] at (B) {1};
  \node[left] at (0,3) {1};
\draw[-] (-0.05,3) -- (0.05,3);
  % Segments
  \coordinate (R) at (3,1.9);
  \coordinate (Bl) at (3,1.3);

  % Zone grisée
  \fill[gray!30] (A) -- (R) -- (Bl) -- cycle;
  \node at (2.6,1.4) {\small $\bar{D}_{ij}$};

  % Segment rouge
  \draw[thick,red] (A) -- (R) node[midway,above,sloped] {$\mu_i \xi - \mu_j y = 0$};

  % Segment bleu
  \draw[thick,blue] (A) -- (Bl) node[midway,below,sloped] {$\bar{\mu}_i \xi - \bar{\mu}_j y = 0$};

  % Accolade
  \draw[thick,orange] (R) -- (Bl) node[midway,right] {$\dfrac{\mu_j}{\mu_i}-\dfrac{\bar{\mu}_j}{\bar{\mu}_i}<C_4\delta$};
  
   % \draw [decorate,decoration={brace,amplitude=6pt,mirror}] (3,1.3) -- (3,1.7)
   %   node[midway,xshift=1.2cm] {$\dfrac{\mu_j}{\mu_i}-\dfrac{\bar{\mu}_j}{\bar{\mu}_i}$};

\end{tikzpicture}}
\end{center}
\caption{Schematic representation of the domain $\bar{D}_{ij}$.}
\label{Fig_D_Delta}
\end{figure*}
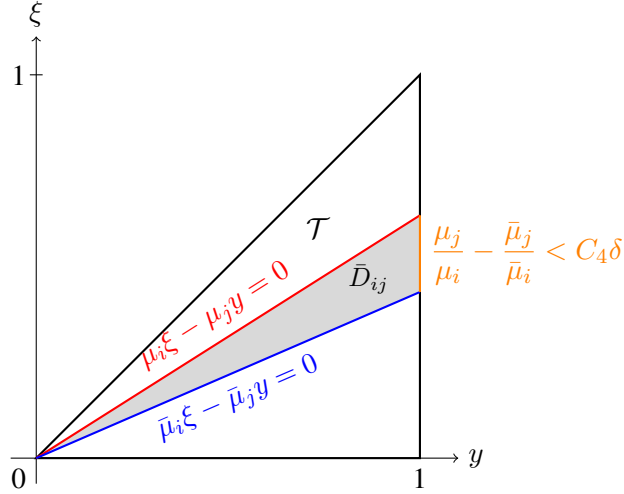

\paragraph*{Analysis of $\tilde{\Phi}(\mathbf{K})$} \label{Sec_Analysis_tilde_phi}
Assume that $\bar{\lambda}_j+\bar{\mu}_i> \lambda_j+\mu_i$ (the result can be easily adjusted for the other case). For all $1\leq i\leq \ell$, all $1\leq j\leq k$, $\psi_{ij}$, and all $(x,y)\in\mathcal{T}$, we have

\begin{multline*}
    \tilde \Phi_{ij}(\mathbf{K})(y,\xi)=\int_0^{\frac{y-\xi}{\bar{\lambda}_j+\bar{\mu}_i}} \Big[\sum_{r=1}^k\bar{\Sigma}_{rj}^{++}K^{-+}_{ir}(y-\bar{\mu}_is,\xi+\bar{\lambda}_js)\\+\sum_{r=1}^\ell\bar{\Sigma}_{rj}^{-+}
 K^{--}_{ir}(y-\bar{\mu}_is,\xi+\bar{\lambda}_js)\Big] \,ds-\int_0^{\frac{y-\xi}{\lambda_j+\mu_i}} \Big[\sum_{r=1}^k\Sigma_{rj}^{++}\\K^{-+}_{ir}(y-\mu_is,\xi+\lambda_js)\nonumber 
 +\sum_{r=1}^\ell\Sigma_{rj}^{-+}K^{--}_{ir}(y-\mu_is,\xi+\lambda_js)\Big] \,ds,\\
 \leq C_5\int_0^{\frac{y-\xi}{\lambda_j+\mu_i}}|\mathds{1}_{[0,\frac{y-\xi}{\bar{\lambda}_j+\bar{\mu}_i}]}(s)-1|ds+|\int_0^{\frac{y-\xi}{\lambda_j+\mu_i}} \Big[\sum_{r=1}^k\bar{\Sigma}_{rj}^{++} \\
 K^{-+}_{ir}(y-\bar{\mu}_is,\xi+\bar{\lambda}_js)+\sum_{r=1}^\ell\bar{\Sigma}_{rj}^{-+}K^{--}_{ir}(y-\bar{\mu}_is,\xi+\bar{\lambda}_js)\nonumber \\
-\Big[\sum_{r=1}^k\Sigma_{rj}^{++}K^{-+}_{ir}(y-\mu_is,\xi+\lambda_js)+\sum_{r=1}^\ell\Sigma_{rj}^{-+}\\
K^{--}_{ir}(y-\mu_is,\xi+\lambda_js)\Big] \,ds|.
\end{multline*}
 Since the kernels $K^{-+}$ and $K^{--}$ are piecewise continuously differentiable, direct computations yield  
\begin{equation}\label{eq:phi_bound}
    \tilde{\Phi}_{ij}(\mathbf{K})(y,\xi) \;\leq\; C_6\delta,
\end{equation}
%with $\nu\bigl((D^1_\delta)_{ij}\bigr) \leq C \delta$.  
Analogous calculations show that the same bound holds for $\tilde{\Phi}_{m+i,j+n}(\mathbf{K})(x,y)$, for all $1 \leq i,j \leq \ell$.

\paragraph*{Analysis of $\tilde{\mathbf{K}}$} We now analyze the behavior of the quantity $\tilde{\mathbf{K}}$. 
Following the method of successive approximations, we introduce the sequence $\tilde{\mathbf{K}}^q$ defined by $\tilde{\mathbf{K}}^0=0$ and, for all $q\geq 1$, 
\begin{equation}
    \tilde{\mathbf{K}}^{q}=\tilde \varphi+\bar{\Phi}(\tilde{\mathbf{K}}^{q-1}).
\end{equation}
For each $q\geq 1$, we define the increment $\Delta \tilde{\mathbf{K}}^q=\tilde{\mathbf{K}}^q-\tilde{\mathbf{K}}^{q-1}$, with $\Delta \tilde{\mathbf{K}}^0=\tilde \varphi$. 
Hence, for all $(y,\xi)\in\mathcal{T}$, we have
\begin{equation}
    \tilde{\mathbf{K}}(y,\xi)=\sum_{q=0}^\infty \Delta \tilde{\mathbf{K}}^q(y,\xi).
\end{equation}

In particular, $\Delta \tilde{\mathbf{K}}^1=\bar{\Phi}(\tilde \varphi)$. 
Adjusting the computations done above, it follows that for all $1\leq i \leq k\ell+\ell^2$, and for all $(y,\xi)\in\mathcal{T}$,
\begin{equation}
   |\Delta \tilde{\mathbf{K}}_i^1(y,\xi)| \leq C_6\delta.%+\mathds{1}_{D^1_\delta}(y,\xi))\bigr).  
\end{equation}
%with $\nu\bigl(D^1_\delta\bigr) \leq C \delta$. 
We recall the following lemma from \cite[Proposition~6.1]{hu2015control}: 

\begin{lemma}\label{lemma:rec_bound}
   Let $q\geq 1$. Assume that for all $(y,\xi)\in\mathcal{T}$ and for all $1\leq i\leq k\ell+\ell^2$,
   \begin{equation*}
       |\Delta \tilde{\mathbf{K}}_i^q(y,\xi)| \leq C\delta M^q\frac{(y-(1-\epsilon_0)\xi)^q}{q!},
   \end{equation*}
   where $|\Delta \tilde{\mathbf{K}}_i^q(y,\xi)|$ denotes the $i^{\text{th}}$ component of $\Delta \tilde{\mathbf{K}}^q(y,\xi)$, $C>0$, and where $\epsilon_0$ and $M$ are two appropriately chosen constants. 
   Then, 
    \begin{equation*}
       |\Delta \tilde{\mathbf{K}}_i^{q+1}(y,\xi)| \leq C\delta M^{q+1}\frac{(y-(1-\epsilon_0)\xi)^{q+1}}{(q+1)!}.
   \end{equation*}
\end{lemma}

We have 
\begin{equation}
   |\tilde{\mathbf{K}}_i(y,\xi)|=\left|\sum_{q=0}^\infty\Delta \tilde{\mathbf{K}}_i^q(y,\xi)\right| 
   \leq \left|\sum_{q=1}^\infty\Delta \tilde{\mathbf{K}}_i^q(y,\xi)\right|+|\tilde \varphi_i|.
\end{equation}
Consequently, combining Lemma~\ref{lemma:rec_bound} with \eqref{eq:psi_bound}, \eqref{eq:gamma_bound}, and \eqref{eq:phi_bound}, we obtain for all $1\leq i \leq \ell$ and $1\leq j \leq k$
\begin{align*}
   & |\bar{K}^{-+}_{ij}(y,\xi)-K_{ij}^{-+}(y,\xi)| \leq  C_7 \big(\delta+\mathds{1}_{\bar{D}^{+-}}(y,\xi)\big),
   \end{align*}
   and for all $1\leq i \leq \ell$ and $1\leq j \leq \ell$
   \begin{align*}
    &|\bar{K}^{--}_{ij}(y,\xi)-K_{ij}^{--}(y,\xi)| \leq  C_8 \big(\delta+\mathds{1}_{\bar{D}^{--}}(y,\xi)\big),
\end{align*}
where $\bar{D}^{+-}$ and $\bar{D}^{--}$ are subsets of $\mathcal{T}$ satisfying $\nu(\bar{D}^{+-})\leq C_9\delta$ and $\nu(\bar{D}^{--})\leq C_{10}\delta$. Similar inequalities can be obtained for $K^{++}$ and $K^{+-}$, establishing~\eqref{eq:ineg_K}, which completes the proof of Lemma~\ref{lemma:convergence_K}.

\subsubsection{Impact of model mismatch on the kernel of the distributed delay in IDE (\ref{sys-final})}
The proof of Lemma~\ref{lemma:convergence_w} is based on analyzing the effect of model mismatch on the kernel $w_1$. {For a given set of parameters $\mathcal{P}$ of \eqref{eq:hyperbolic_couple} satisfying \eqref{eq:original_order}, let $\delta\in[0,\hat \delta]$, with $\hat{\delta}\in(0,\delta^{\star})$ with $\delta^{\star}$ defined in Proposition \ref{prop:closed_ball} and $\mathcal{\bar P}\in\mathcal{B}_{\mathcal{P}}^\delta$. }
We first examine its impact on the kernel $L$ and on the functions $G^{+-}$, $G^{-+}$, $N^+$, $N^-$, and the kernel $F$. 
In what follows, we denote with a bar these functions when they are computed using the parameters 
$\bar{\Lambda}$, $\bar{\Sigma}$, $\bar{Q}$, and $\bar{R}$.

\paragraph*{Effect of model mismatch on the kernel $L$}
The kernel \(L\) is defined by the Volterra equation~\eqref{eq:def_L}. 
It can be shown~\cite{auriol2024contributions} that it satisfies kernel equations analogous to those of $K$. 
Therefore, the computations presented above can be adapted to obtain, for all $(y,\xi)\in \mathcal{T}$, for all $1\leq i,j \leq k+\ell$
\begin{equation}
   |L_{ij}(y,\xi)-\bar{L}_{ij}(y,\xi)|\leq C_{11}\delta+C\mathds{1}_{\bar{D}}(y,\xi),
\end{equation}
where $\bar{D}$ is a subset of $\mathcal{T}$ such that $\nu\bar{D}\leq C_{12}\delta$. 
As before, the notation $\bar{D}$ will be overloaded in the sequel.

\paragraph*{Effect of model mismatch on the functions $G^{-+}$, $G^{--}$, $N^+, N^-$}
The functions $G^{-+}$, $G^{--}$, $N^+$, and $N^-$ are directly obtained from the kernels $K$ and $L$ through 
equations~\eqref{eq:def_G_-+}--\eqref{eq:def_N_-}. 
Since these functions are piecewise continuous, it follows immediately that, for all $y\in[0,1]$,  and all $1\leq i\leq k,~1\leq j \leq \ell$
\begin{equation*}
   |G^{+-}(y)-\bar{G}^{+-}_{ij}(y)|\leq C_{13}(\delta+\mathds{1}_{\bar{D}}(0,y)),  
\end{equation*}
\begin{equation*}
   |N_{ij}^{+}(y)-\bar{N}^{+}_{ij}(y)|\leq C_{15}(\delta+\mathds{1}_{\bar{D}}(1,y)), 
\end{equation*}
and for all $1\leq i,j \leq \ell$
\begin{equation*}
   |G_{ij}^{--}(y)-\bar{G}^{--}_{ij}(y)|\leq C_{14}(\delta+\mathds{1}_{\bar{D}}(0,y)), 
\end{equation*}
\begin{equation*}
   |N_{ij}^{-}(y)-\bar{N}^{-}_{ij}(y)|\leq C_{16}(\delta+\mathds{1}_{\bar{D}}(1,y)).
\end{equation*}

\paragraph*{Effect of model mismatch on the kernel $F$}
The kernel $F$ is computed from the functions $G^{--}$ by applying the method of characteristics 
to equations~\eqref{eq:F_PDE}--\eqref{eq:F_BC}. 
Straightforward calculations yield, for all $(y,\xi)\in \mathcal{T}$, 
\begin{equation*}
   |F_{ij}(y,\xi)-\bar{F}_{ij}(y,\xi)|\leq C_{17}(\delta+\mathds{1}_{\bar{D}}(y,\xi)),~1\leq i,j\leq \ell.
\end{equation*}

\paragraph*{Effect of model mismatch on the kernel $w_1$}
Finally, adapting the computations of~\cite{Auriol2019a}, we can derive the integral equation satisfied by 
$w_-(t,1)$ in the presence of model mismatch. We immediately obtain equation~\eqref{eq:ineg_w}.

\bibliographystyle{plain}
\bibliography{ref}

@article{di2018stabilization,
  title={Stabilization of coupled linear heterodirectional hyperbolic {PDE--ODE} systems},
  author={Di Meglio, F. and Bribiesca  Argomedo, F. and Hu, L. and Krstic, M.},
  journal={Automatica},
  volume={87},
  pages={281--289},
  year={2018},
  publisher={Elsevier}
}

@Book{yoshida1960lectures,
  Title                    = {Lectures on differential and integral equations},
  Author                   = {Yoshida, K.},
  Publisher                = {Interscience Publishers},
  Year                     = {1960},
  Volume                   = {10}
}

@article{hu2015control,
  title={Control of homodirectional and general heterodirectional linear coupled hyperbolic {PDE}s},
  author={Hu, L. and Di Meglio, F. and Vazquez, R. and Krstic, M.},
  journal={IEEE Transactions on Automatic Control},
  volume={61},
  number={11},
  pages={3301--3314},
  year={2015},
  publisher={IEEE}
}

@Article{auriol2018delay,
  Title                    = {Delay-robust stabilization of a hyperbolic {P}{D}{E}--{O}{D}{E} system},
  Author                   = {Auriol, J. and Bribiesca Argomedo, F. and Bou Saba, D. and Di Loreto, M. and Di Meglio, F.},
  Journal                  = {Automatica},
  Year                     = {2018},
  Pages                    = {494--502},
  Volume                   = {95},

  Publisher                = {Elsevier}
}

@article{auriol2020robust,
  title={Robust output feedback stabilization for two heterodirectional linear coupled hyperbolic {PDEs}},
  author={Auriol, J. and Di Meglio, F.},
  journal={Automatica},
  volume={115},
  pages={108896},
  year={2020},
  publisher={Elsevier}
}

@article{hu2019boundary,
  title={Boundary exponential stabilization of 1-dimensional inhomogeneous quasi-linear hyperbolic systems},
  author={Hu, L. and Vazquez, R. and Di Meglio, F. and Krstic, M.},
  journal={SIAM Journal on Control and Optimization},
  volume={57},
  number={2},
  pages={963--998},
  year={2019},
  publisher={SIAM}
}

@article{coron2017finite,
  title={Finite-time boundary stabilization of general linear hyperbolic balance laws via {Fredholm} backstepping transformation},
  author={Coron, J.-M. and Hu, L. and Olive, G.},
  journal={Automatica},
  volume={84},
  pages={95--100},
  year={2017},
  publisher={Elsevier}
}

@article{ prevpaper,
  title={On the filtered spectral abscissa of delay-difference equations and its role in the boundary control of hyperbolic PDEs},
  author={Michiels, W. and Auriol, J. and Bribiesca-Argomedo, F.},
  journal="{SIAM} Journal on Control and Optimization",
note = "Submitted",
  year={2025},
}

@article{auriol2020closed,
  title={Closed-loop tool face control with the bit off-bottom},
  author={Auriol, J. and Shor, R. J. and Aarsnes, U. J. F. and Di Meglio, F.},
  journal={Journal of Process Control},
  volume={90},
  pages={35--45},
  year={2020},
  publisher={Elsevier}
}

@article{auriol2022comparing,
  title={Comparing advanced control strategies to eliminate stick-slip oscillations in drillstrings},
  author={Auriol, J. and Boussaada, I. and Shor, R. J. and Mounier, H. and Niculescu, S.-I.},
  journal={IEEE Access},
  volume={10},
  pages={10949--10969},
  year={2022},
  publisher={IEEE}
}

@article{diagne2017control,
  title={Control of shallow waves of two unmixed fluids by backstepping},
  author={Diagne, M. and Tang, S-X. and Diagne, A. and Krstic, M.},
  journal={Annual Reviews in Control},
  volume={44},
  pages={211--225},
  year={2017},
  publisher={Elsevier}
}

@article{saba2019stability,
  title={{Stability Analysis for a Class of Linear $2\times 2$ Hyperbolic {PDE}s Using a Backstepping Transform}},
  author={Bou Saba, D.  and Bribiesca-Argomedo, F. and Auriol, J. and Di Loreto, M. and Di Meglio, F.},
  journal={IEEE Transactions on Automatic Control},
  volume={65},
  number={7},
  pages={2941--2956},
  year={2019},
  publisher={IEEE}
}

@Book{auriol2024contributions,
  author    = {Auriol, J.},
  title     = {Contributions to the robust stabilization of networks of hyperbolic systems},
  publisher = {HDR, Universit{\'e} Paris Saclay},
  year      = {2024},
}

@article{schunk1975transport,
  title={Transport equations for aeronomy},
  author={Schunk, R},
  journal={Planetary and Space Science},
  volume={23},
  number={3},
  pages={437--485},
  year={1975},
  publisher={Elsevier}
}

@article{espitia2022traffic,
  title={Traffic flow control on cascaded roads by event-triggered output feedback},
  author={Espitia, N. and Auriol, J. and Yu, H. and Krstic, M.},
  journal={International Journal of Robust and Nonlinear Control},
  volume={32},
  number={10},
  pages={5919--5949},
  year={2022},
  publisher={Wiley Online Library}
}

@book{bastin2016stability,
  title={Stability and boundary stabilization of 1-d hyperbolic systems},
  author={Bastin, G and Coron, J-M},
  volume={88},
  year={2016},
  publisher={Springer}
}

@article{hayat2021boundary,
  title={Boundary stabilization of {1D} hyperbolic systems},
  author={Hayat, A.},
  journal={Annual Reviews in Control},
  volume={52},
  pages={222--242},
  year={2021},
  publisher={Elsevier}
}

@article{vazquez2024backstepping,
  title={Backstepping for Partial Differential Equations},
  author={Vazquez, R. and Auriol, J. and Bribiesca-Argomedo, F. and Krstic, M.},
  journal={Automatica},
  year={2025}
}

@Article{DAlembert,
  author    = {D'Alembert, J.},
  journal   = {Histoire de l'Académie Royale des Sciences et des Belles Lettres de Berlin},
  title     = {Suite des recherches sur la courbe que forme une corde tendue, mise en vibration},
  year      = {1749},
  pages     = {220-249},
}

@Article{russell1991neutral,
  author    = {Russell, D.L.},
  title     = {Neutral {PDE} canonical representations of hyperbolic systems},
  pages     = {129--166},
  journal   = {The Journal of Integral Equations and Applications},
  publisher = {JSTOR},
  year      = {1991},
}

@Article{karafyllis2014relation,
  author    = {Karafyllis, I. and Krstic, M.},
  journal   = {ESAIM: Control, Optimisation and Calculus of Variations},
  title     = {On the relation of delay equations to first-order hyperbolic partial differential equations},
  year      = {2014},
  issn      = {1262-3377},
  number    = {3},
  pages     = {894--923},
  volume    = {20},
  publisher = {EDP Sciences},
}

@Article{Auriol2019a,
  author    = {Auriol, J. and Di Meglio, F.},
  title     = {An explicit mapping from linear first order hyperbolic {P}{D}{E}s to difference systems},
  pages     = {144--150},
  volume    = {123},
  journal   = {Systems \& Control Letters},
  publisher = {Elsevier},
  year      = {2019},
}

@Article{Auriol2023,
  author    = {Auriol, J. and Bribiesca-Argomedo, F. and Di Meglio, F.},
  title     = {Robustification of stabilizing controllers for {ODE}-{PDE}-{ODE} systems: a filtering approach},
  pages     = {110724},
  volume    = {147},
  journal   = {Automatica},
  publisher = {Elsevier},
  year      = {2023},
}

@ARTICLE{ wimfilter,
	author={Michiels, W.},
	journal={IEEE Transactions on Automatic Control}, 
	title={To Filter or Not to Filter? Impact on Stability of Delay-Difference and Neutral Equations With Multiple Delays}, 
	year={2023},
	volume={68},
	number={6},
	pages={3687-3693},
}

@article{wimneutral2,
author = {Michiels, W. and Vyhl\'idal, T. and Zit\'ek, P. and Nijmeijer, H. and Henrion, D.},
title = "Strong stability of neutral equations with an arbitrary delay dependency structure",
journal ="{SIAM} Journal on Control and Optimization",
volume = "48",
number = "2",
pages = "763-786",
year = "2009"
}

@article { extraH,
	author = "Logemann, H. and Townley, S.",
	title = "The effect of small delays in the feedback loop on the stability of neutral
	systems",
	journal = "Systems \& Control Letters",
	volume = "27",
	year = "1996",
	pages = "267-274",
}

@book{ bookwim,
	address = {Philadelphia, PA},
	author = {Michiels, W. and Niculescu, S.-I.},
	publisher = {Society for Industrial and Applied Mathematics},
	title = {{Stability, Control, and Computation for Time-Delay Systems}},
	year = {2014},
}

@article{ w-stable,
	title = {w-{Stability} of feedback systems},
	author = {Georgiou, T.T. and Smith, M.},
	journal = {Systems \& Control Letters},
	volume = {13},
	number = {4},
	pages = {271-277},
	year = {1989},
	issn = {0167-6911},
}

@article {have:02,
	author= "Hale, J. K. and Verduyn Lunel, S. M.",
	title=" Strong stabilization of neutral functional differential equations",
	journal="IMA Journal of Mathematical Control and Information",
	volume="19",
	pages="5-23",
	year="2002",
}

\ifCLASSOPTIONcaptionsoff
  \newpage
\fi

\end{document}